\documentclass[1p,sort&compress]{elsarticle}
\usepackage{latexsym,bm,stmaryrd}
\usepackage{amsmath,amsthm,amsfonts,amssymb,mathrsfs,pb-diagram}
\usepackage{tikz}
\usetikzlibrary{arrows,decorations.markings}

\advance\textheight by 1.5mm

\usepackage[enableskew,vcentermath]{youngtab}
\def\tab(#1){\mbox{\small$\young(#1)$}\,}

\let\<=\langle
\let\>=\rangle
\def\({\big(}
\def\){\big)}
\def\Prod{\displaystyle\prod}
\def\Sum{\displaystyle\sum}

\def\pmod#1{\text{ }(\text{mod } #1)\,}

\newcommand{\BS}{\mathfrak S}
\let\Sym=\BS
\let\Sect=\S
\newcommand{\N}{\mathbb N_0}
\newcommand{\Z}{\mathbb Z}
\def\P{\mathscr{P}}
\newcommand{\HH}{\mathscr{H}}
\renewcommand{\H}[1][n]{\HH^\Lambda_{#1}}

\DeclareMathAlphabet{\mathpzc}{OT1}{pzc}{m}{it}
\def\K{\mathpzc K}
\def\O{\mathpzc O}
\newcommand\HO[1][n]{\HH^\O_{#1}}
\newcommand\HK[1][n]{\HH^\K_{#1}}
\def\Hbeta{{\H[\beta]}}
\renewcommand{\L}[1][n]{\mathscr L^\Lambda_{#1}}
\def\m{\mathfrak{m}}
\newcommand{\R}[1][n]{\mathscr R^\Lambda_{#1}}
\newcommand{\bR}[1][n]{\hat{\mathscr R}^\Lambda_{#1}}
\newcommand\Hblam[1][\blam]{\HH_n^{\gdom #1}}
\newcommand\Hpblam[1][\blam]{\HH_n^{\prime\gdom #1}}

\renewcommand\O{{\mathcal{O}}}
\newcommand{\Q}{\mathbb Q}
\newcommand{\Multiparts}[1][n]{\P^\Lambda_{#1}}

\newcommand{\bi}{\mathbf{i}}
\newcommand{\bj}{\mathbf{j}}
\def\eblam{e_\blam}
\def\ebllam{e_\blam'}
\def\ilam{\bi^\blam}

\def\tlam{{\mathfrak t}^\blam}
\def\tlamp{{\mathfrak t}^\blamp}
\def\tllam{{\mathfrak t}_\blam}
\def\tllamp{{\mathfrak t}_\blamp}

\newcommand{\dns}[1][n]{d^{\eps,s}_{#1}}
\newcommand{\ins}[1][n]{\bi^{\eps,s}_{#1}}
\newcommand{\zns}[1][n]{z^{\eps,s}_{#1}}

\def\a{{\mathfrak a}}
\def\b{{\mathfrak b}}
\def\c{{\mathfrak c}}
\def\d{{\mathfrak d}}
\def\s{{\mathfrak s}}
\def\t{{\mathfrak t}}
\def\u{{\mathfrak u}}
\def\v{{\mathfrak v}}

\newcommand\Dec[1][A]{\mathbf{D}_{#1}(t)}
\newcommand\Cart[1][A]{\mathbf{C}_{#1}(t)}

\newcommand{\lam}{\lambda}

\newcommand{\eps}{\varepsilon}

\newcommand\blam{{\boldsymbol\lambda}}
\newcommand\blamp{{\boldsymbol\lambda'}}

\newcommand\bmu{{\boldsymbol\mu}}
\newcommand\bnu{{\boldsymbol\nu}}

\def\Add{\mathscr A}
\def\LAdd{\Add^\Lambda}
\def\Rem{\mathscr R}
\def\LRem{\Rem^\Lambda}

\DeclareMathOperator\Shape{Shape}
\DeclareMathOperator\Std{Std}
\def\SStd(#1,#2){\Std^\Lambda_{#2}(#1)}

\renewcommand\t{\mathfrak{t}}
\renewcommand\u{\mathfrak{u}}

\let\gedom=\trianglerighteq
\let\gdom=\vartriangleright

\newcommand\Dim[2][t]{\text{\rm Dim}_{#1}#2}

\newcommand\rest[2]{#1_{#2}}
\DeclareMathOperator{\defect}{def}
\DeclareMathOperator{\cont}{cont}
\DeclareMathOperator{\res}{res}
\DeclareMathOperator{\rad}{rad}
\DeclareMathOperator{\Rad}{Rad}
\DeclareMathOperator{\rank}{rank}

\newcommand{\bQ}{\mathbf{Q}}

\newcommand{\charge}{\boldsymbol{\kappa}_\Lambda}

\DeclareMathOperator{\Hom}{Hom}
\def\ZHom{\text{\rm Hom}^\Z}

\def\Mod{\textbf{-Mod}}

\newcounter{main}
\theoremstyle{plain}

\newtheorem*{THEOREM}{Main Theorem}
\swapnumbers \numberwithin{equation}{section}
\newtheorem{Prop}[equation]{Proposition}
\newtheorem{Theorem}[equation]{Theorem}
\newtheorem{cor}[equation]{Corollary}
\newtheorem{Lemma}[equation]{Lemma}
\theoremstyle{definition}
\newtheorem{Defn}[equation]{Definition}
\theoremstyle{remark}
\newtheorem{Remark}[equation]{Remark}

\newenvironment{Example}[1][\relax]%
  {\refstepcounter{equation}\trivlist
   \item[\hskip\labelsep\theequation.~\textbf{Example#1}\space]
   \ignorespaces
  }{\unskip\nobreak\hfil%
    \penalty50\hskip2em\hbox{}\nobreak\hfil$\Diamond$%
    \parfillskip=0pt\finalhyphendemerits=0\penalty-100\endtrivlist}

\def\map#1#2{\,{:}\,#1\!\longrightarrow\!#2}

{\catcode`\|=\active
  \gdef\set#1{\mathinner{\lbrace\,{\mathcode`\|"8000%
                                   \let|\midvert #1}\,\rbrace}}
}
\def\midvert{\egroup\mid\bgroup}

\begin{document}
\title%[Graded cellular algebras]%
{Graded cellular bases for the cyclotomic Khovanov-Lauda-Rouquier algebras of type $A$}

  \author[JH]{Jun Hu}
  \ead[JH]{junhu303@yahoo.com.cn}
  \author[AM]{Andrew Mathas}
  \ead[AM]{a.mathas@usyd.edu.au}
  \address{School of Mathematics and Statistics F07,
  University of Sydney, NSW 2006, Australia}

\begin{abstract}This paper constructs an explicit homogeneous cellular basis for
  the cyclotomic Khovanov--Lauda--Rouquier algebras of type~$A$.
\end{abstract}

\maketitle

%%%%%%%%%%%%%%%%%%%%%%%%%%%%%%%%%%%%%%%%%%%%%%%%%%%%%%%%
\section{Introduction}
In a groundbreaking series of papers Brundan and Kleshchev (and
Wang)~\cite{BK:GradedKL,BKW:GradedSpecht,BK:GradedDecomp} have shown that
the cyclotomic Hecke algebras of type $G(\ell,1,n)$, and their rational
degenerations, are graded algebras. Moreover, they have extended Ariki's
categorification theorem~\cite{Ariki:can} to show over a field of
characteristic zero the graded decomposition numbers of these algebras can
be computed using the canonical bases of the higher level Fock spaces.

The starting point for Brundan and Kleshchev's work was the introduction of
certain graded algebras $\R$ which arose from Khovanov and
Lauda's~\cite[\Sect3.4]{KhovLaud:diagI} categorification of the negative
part of quantum group of an arbitrary Kac-Moody Lie algebra and,
independently, in work of Rouquier~\cite{Rouq:2KM}. In type~A Brundan and
Kleshchev~\cite{BK:GradedKL} proved that the (degenerate and
non-degenerate) cyclotomic Hecke algebras are $\Z$-graded by
constructing explicit isomorphisms to~$\R$.

The \textbf{cyclotomic Khovanov-Lauda--Rouquier algebra} $\R$ is generated
by certain elements
$\{\psi_1,\dots,\psi_{n-1}\}\cup\{y_1,\dots,y_n\}\cup
     \set{e(\bi)|\bi\in(\Z/e\Z)^n}$
which are subject to a long list of relations (see
Definition~\ref{relations}). Each of these relations is homogeneous, so it
follows directly from the presentation that $\R$ is
$\Z$-graded. Unfortunately, it is not at all clear from the relations
how to construct a homogeneous basis of $\R$, even using the
isomorphism from $\R$ to the cyclotomic Hecke algebras.

The main result of this paper gives an explicit homogeneous basis of
$\R$. In fact, this basis is cellular so our Main Theorem
also proves a conjecture of Brundan, Kleshchev and
Wang~\cite[Remark~4.12]{BKW:GradedSpecht}.

To describe this basis let $\Multiparts$ be the set of multipartitions of
$n$, which is a poset under the dominance order. For each
$\blam\in\Multiparts$ let $\Std(\blam)$ be the set of standard
$\blam$-tableaux (these terms are defined in \Sect3.3). For each
$\blam\in\Multiparts$ there is an idempotent $\eblam$ and a homogeneous
element $y_\blam\in K[y_1,\dots,y_n]$ (see Definition~\ref{ylam defn}).  
Brundan, Kleshchev and Wang~\cite{BKW:GradedSpecht}
have defined a combinatorial \textit{degree} function
$\deg\map{\coprod_\blam\Std(\blam)}\Z$ and for each $\t\in\Std(\blam)$
there is a well-defined element
$\psi_{d(\t)}\in\<\psi_1,\dots,\psi_{n-1}\>$ and we set
$\psi_{\s\t}=\psi_{d(\s)^{-1}}e_\blam y_\blam \psi_{d(\t)}$. 
Our Main Theorem is the following.

\begin{THEOREM}
  Suppose that $\O$ is a commutative integral domain such that $e$ is
  invertible in $\O$, $e=0$, or $e$ is a non-zero prime number, and let $\R$ be the cyclotomic Khovanov-Lauda--Rouquier
  algebra~$\R$ over~$\O$. Then $\R$ is a graded cellular algebra
  with respect to the dominance order and with homogeneous cellular basis
  $$\set{\psi_{\s\t}|\blam\in\Multiparts \text{ and }\s,\t\in\Std(\blam)}.$$
  Moreover, $\deg\(\psi_{\s\t}\)=\deg\s+\deg\t$.
\end{THEOREM}

We prove our Main Theorem by considering the two really interesting cases
where $\R$ is isomorphic to either a degenerate or a non-degenerate
cyclotomic Hecke algebra over a field. In these two cases we
show that $\{\psi_{\s\t}\}$ is a homogeneous cellular basis of $\R$. We
then use these results to deduce our main theorem

The main difficulty in proving this theorem is that the graded presentation
of the cyclotomic Khovanov-Lauda--Rouquier algebras hides many of the
relations between the homogeneous generators. We overcome this by first
observing that the KLR idempotents $e(\bi)$, for $\bi\in I^n$, are precisely
the primitive idempotents in the subalgebra of the cyclotomic Hecke algebra
which is generate by the Jucys-Murphy elements (Lemma~\ref{idempotents}).
Using results from~\cite{M:seminormal} this allows us to lift $e(\bi)$ to
an element $e(\bi)^\O$ which lives in an integral form of the Hecke algebra
defined over a suitable discrete valuation ring $\O$. The elements
$e(\bi)^\O$ can be written as natural linear combinations of the seminormal
basis elements~\cite{M:gendeg}. In turn this allows us to construct a
family of non-zero elements $\eblam y_\blam$, for $\blam$ a multipartition,
which form the skeleton of our cellular basis and hence prove our main
theorem.

In fact, we give two graded cellular bases of the cyclotomic
Khovanov-Lauda-Rouquier algebras~$\R$.  Intuitively, one of these bases is
built from the \textit{trivial} representation of~the Hecke algebra and the
other is built from its \textit{sign} representation. We then show
that these two bases are dual to each other, modulo more dominant terms. As
a consequence, we deduce that the blocks of $\R$ are graded symmetric
algebras (see Corollary~\ref{gradsy}), as conjectured by Brundan and
Kleshchev\cite[Remark~4.7]{BK:GradedDecomp}.

This paper is organized as follows. In section~2 we define and develop
the representation theory of \textit{graded cellular algebras},
following and extending ideas of Graham and Lehrer~\cite{GL}. Just as
with the original definition of cellular algebras, graded cellular
algebras are already implicit in the literature in the work of Brundan
and Stroppel~\cite{BrundanStroppel:KhovanovI,BrundanStroppel:KhovanovIII}.
In section~3, following Brundan and Kleshchev~\cite{BK:GradedKL} we
define the cyclotomic Khovanov-Lauda--Rouquier algebras of type
$G(\ell,1,n)$ and recall Brundan and Kleshchev's all important graded
isomorphism theorem. In section~4 we shift gears and show how to lift
the idempotents $e(\bi)$ to~$\HO$, an integral form of the
non-degenerate cyclotomic Hecke algebra~$\H$. We then use this
observation to produce a family of non-trivial homogeneous elements
of~$\R\cong\H$, including $\eblam y_\blam$, for $\blam\in\Multiparts$.
In section~5 we lift the graded Specht modules of Brundan, Kleshchev
and Wang to give a graded basis of $\H$ and then in section~6 we
construct the dual graded basis and use this to show that the blocks
of $\H$ are graded symmetric algebras. As an application we construct
an isomorphism between the graded Specht modules and the dual of the
dual graded Specht modules, which are defined using our second graded
cellular basis of~$\H$. In an appendix, which was actually the
starting point for this work, we use a different approach to
explicitly describe the homogeneous elements which span the one
dimensional two-sided ideals of~$\H$.

%%%%%%%%%%%%%%%%%%%%%%%%%%%%%%%%%%%%%%%%%%%%%%%%%%%%%%%%
\section{Graded cellular algebras} This section defines graded cellular
algebras and develops their representation theory, extending Graham and
Lehrer's~\cite{GL} theory of cellular algebras.  Most of the arguments of
Graham and Lehrer apply with minimal change in the graded setting. In
particular, we obtain graded cell modules, graded simple and projective
modules and a graded analogue of Brauer-Humphreys reciprocity.

\subsection{Graded algebras}
Let $R$ be a commutative integral domain with~$1$. In this paper a
\textbf{graded $R$-module} is an $R$-module $M$ which has a direct sum
decomposition $M=\bigoplus_{d\in\Z}M_d$. If  $m\in M_d$, for $d\in\Z$,
then $m$ is \textbf{homogeneous} of \textbf{degree} $d$ and we set
$\deg m=d$. If $M$ is a graded $R$-module let $\underline{M}$ be the
ungraded $R$-module obtained by forgetting the grading on~$M$. If $M$
is a graded $R$-module and $s\in\Z$ let $M\<s\>$ be the graded
$R$-module obtained by shifting the grading on $M$ up by $s$; that is, 
$M\<s\>_d=M_{d-s}$, for $d\in\Z$.

A \textbf{graded $R$-algebra} is a unital associative $R$-algebra
$A=\bigoplus_{d\in\Z}A_d$ which is a graded $R$-module such that
$A_dA_e\subseteq A_{d+e}$, for all $d,e\in\Z$. It follows that $1\in
A_0$ and that $A_0$ is a graded subalgebra of $A$.  A graded (right)
$A$-module is a graded $R$-module $M$ such that $\underline{M}$ is 
an $\underline{A}$-module and $M_dA_e\subseteq M_{d+e}$, for all
$d,e\in\Z$. Graded submodules, graded left $A$-modules and so on are all
defined in the obvious way. Let $A\Mod$ be the category of all
finitely generated graded $A$-modules together with degree preserving
homomorphisms; that is,
$$\Hom_A(M,N)=\set{f\in\Hom_{\underline A}(\underline M,\underline N)|
               f(M_d)\subseteq N_d\text{ for all }d\in\Z},$$
for all $M,N\in A\Mod$. The elements of $\Hom_A(M,N)$ are homogeneous
maps of degree~$0$. More generally, if 
$f\in\Hom_A(M\<d\>,N)\cong\Hom_A(M,N\<-d\>)$ then $f$
is a homogeneous map from $M$ to $N$ of degree $d$ and we write $\deg f=d$. Set
$$\ZHom_A(M,N)=\bigoplus_{d\in\Z}\Hom_A(M\<d\>,N)
              \cong\bigoplus_{d\in\Z}\Hom_A(M,N\<-d\>)
$$
for $M,N\in A\Mod$.

\subsection{Graded cellular algebras}Following Graham and
Lehrer~\cite{GL} we now define graded cellular algebras.

\begin{Defn}[Graded cellular algebras]\label{graded cellular def}
  Suppose that $A$ is a $\Z$-graded $R$-algebra which is free of finite rank over $R$. A
  \textbf{graded cell datum} for $A$ is an ordered quadruple
  $(\P,T,C,\deg)$, where $(\P,\gdom)$ is the \textbf{weight poset},
  $T(\lambda)$ is a finite set for $\lambda\in\P$, and
  $$C\map{\coprod_{\lambda\in\P}T(\lambda)\times T(\lambda)}A;
     (\s,\t)\mapsto c^\lambda_{\s\t},
     \quad\text{and}\quad
     \deg\map{\coprod_{\lambda\in\P}T(\lambda)}\Z$$
  are two functions such that $C$ is injective and
  \begin{enumerate}
    \item[(GC$_d$)] Each basis element $c^\lambda_{\s\t}$ is homogeneous
	of degree $\deg c^\lambda_{\s\t}=\deg\s+\deg\t$, for $\lambda\in\P$ and
      $\s,\t\in T(\lambda)$.
    \item[(GC$_1$)] $\set{c^\lambda_{\s\t}|\s,\t\in T(\lambda), \lambda\in\P}$ is an
      $R$-basis of $A$.
    \item[(GC$_2$)] If $\s,\t\in T(\lambda)$, for some $\lambda\in\P$, and $a\in A$ then
    there exist scalars $r_{\t\v}(a)$, which do not depend on $\s$, such that
      $$c^\lambda_{\s\t} a=\sum_{\v\in T(\lambda)}r_{\t\v}(a)c^\lambda_{\s\v}\pmod
      {A^{\gdom\lambda}},$$
      where $A^{\gdom\lambda}$ is the $R$-submodule of $A$ spanned by
      $\set{c^\mu_{\a\b}|\mu\gdom\lambda\text{ and }\a,\b\in T(\mu)}$.
    \item[(GC$_3$)] The $R$-linear map $*\map AA$ determined by
      $(c^\lambda_{\s\t})^*=c^\lambda_{\t\s}$, for all $\lambda\in\P$ and
      all $\s,\t\in\P$, is an anti-isomorphism of $A$.
  \end{enumerate}
  A \textbf{graded cellular algebra} is a graded algebra which has a graded
  cell datum. The basis $\set{c^\lambda_{\s\t}|\lambda\in\P\text{ and }
  \s,\t\in T(\lambda}$ is a \textbf{graded cellular basis} of~$A$.
\end{Defn}

If we omit (GC$_d$) then we recover Graham and Lehrer's definition of an
(ungraded) cellular algebra. Therefore, by forgetting the grading, any
graded cellular algebra is an (ungraded) cellular algebra in the original
sense of Graham and Lehrer.

\begin{Example}[s] a) Let $A=\mathfrak{gl}_2(R)$ be the algebra of $2\times 2$
  matrices over~$R$. Let $\P=\{*\}$ and $T(*)=\{1,2\}$ and set
  $$c_{11}=e_{12},\quad c_{12}=e_{11}, \quad c_{21}=e_{22}
  \quad\text{and}\quad c_{22}=e_{21},$$
  with $\deg(1)=1$ and $\deg(2)=-1$. Then $(\P,T,C,\deg)$ is a graded
  cellular basis of~$A$. In particular, taking $R$ to be a field this
  shows that semisimple algebras can be given the structure of a
  graded cellular algebra with a non-trivial grading.\\
  b) Brundan has pointed out that it follows from his results with Stroppel
  that the Khovanov diagram algebras\cite[Cor.~3.3]{BrundanStroppel:KhovanovI},
  their quasi-hereditary covers~\cite[Theorem~4.4]{BrundanStroppel:KhovanovI},
  and the level two degenerate cyclotomic Hecke
  algebras~\cite[Theorem~6.6]{BrundanStroppel:KhovanovIII} are all graded
  cellular algebras in the sense of Definition~\ref{graded cellular def}.
\end{Example}

\begin{Defn}[Graded cell modules]\label{graded cells}
    Suppose that $A$ is a graded cellular algebra with graded cell
    datum $(\P,T,C,\deg)$, and fix $\lambda\in\P$. Then the
    \textbf{graded cell module} $C^\lambda$ is the graded right $A$-module
    $$C^\lambda=\bigoplus_{z\in\Z}C^\lambda_z,$$
    where $C^\lambda_z$ is the free $R$-module with basis
    $\set{c^\lambda_\t|\t\in T(\lambda)\text{ and }\deg\t=z}$ and
    where the action of $A$ on $C^\lambda$ is given by
    $$ c^\lambda_\t a=\sum_{\v\in T(\lambda)}r_{\t\v}(a) c^\lambda_\v,$$
    where the scalars $r_{\t\v}(a)$ are the scalars appearing in
    (GC$_2$).

    Similarly, let $C^{*\lambda}$ be the left graded $A$-module which,
    as an $R$-module is equal to $C^\lambda$, but where the $A$-action
    is given by $a\cdot x:= xa^*$, for $a\in A$ and $x\in
    C^{*\lambda}$.
\end{Defn}

It follows directly from Definition~\ref{graded cellular def} that $C^\lambda$
and $C^{*\lambda}$ are graded $A$-modules. Let $A^{\gedom\lambda}$ be the
$R$-module spanned by the elements
$\set{c^\mu_{\u\v}|\mu\gedom\lambda\text{ and }\u,\v\in T(\mu)}$. It
is straightforward to check that
$A^{\gedom\lambda}$ is a graded two-sided ideal of $A$ and that
\begin{equation}\label{two cells}
A^{\gedom\lambda}/A^{\gdom\lambda}\cong C^{*\lambda}\otimes_R C^\lambda
      \cong\bigoplus_{\s\in T(\lambda)}C^\lambda\<\deg\s\>
\end{equation}
as graded $(A,A)$-bimodules for the first isomorphism and as graded
right $A$-modules for the second.

Let $t$ be an indeterminate over $\N$. If $M=\oplus_{z\in\Z}M_z$ is
a graded $A$-module such that each $M_z$ is free of finite rank over
$R$, then its \textbf{graded dimension} is the Laurent  polynomial
$$\Dim M=\sum_{k\in\Z}(\dim_{R}M_k)t^k.$$

\begin{cor}\label{graded dimension}
    Suppose that $A$ is a graded cellular algebra and $\lambda\in\P$. Then
    $$\Dim{C^\lambda}=\sum_{\s\in T(\lambda)}t^{\deg\s}.$$
    Consequently,
    $\Dim A=\Sum_{\lambda\in\P}\sum_{\s,\t\in T(\lambda)} t^{\deg\s+\deg\t}
              =\sum_{\lambda\in\P}\(\Dim{C^\lambda}\)^2$.
\end{cor}

Suppose that $\mu\in\P$. Then it follows from
Definition~\ref{graded cellular def}, exactly as in \cite[Prop.~2.4]{GL},
that there is a bilinear form $\<\ ,\ \>_\mu$ on $C^\mu$ which
is determined by
$$c^\mu_{\a\s}c^\mu_{\t\b}\equiv
  \<c^\mu_\s,c^\mu_\t\>_\mu c^\mu_{\a\b}\pmod{A^{\gdom\mu}},$$
for any $\s,\t, \a,\b\in T(\mu)$. The next Lemma gives standard
properties of this bilinear form $\<\ ,\ \>_\mu$. Just as in the
ungraded case (see, for example, \cite[Prop.~2.9]{M:Ulect}) it follows
directly from the definitions.

\begin{Lemma}\label{symmetric form}
    Suppose that $\mu\in\P$ and that $a\in A$, $x,y\in C^\mu$. Then
    $$\<x,y\>_\mu=\<y,x\>_\mu, \qquad \<xa,y\>_\mu=\<x,ya^*\>_\mu
    \qquad\text{and}\qquad
    xc^\mu_{\s\t}=\<x,c^\mu_\s\>_\mu c^\mu_\t,$$
    for all $\s,\t\in T(\mu)$.
\end{Lemma}

We consider the ring $R$ as a graded $R$-module with trivial grading:
$R=R_0$. Observe that $C^\mu\otimes C^\mu$ is a graded
$A$-module with $\deg x\otimes y=\deg x+\deg y$.

\begin{Lemma}\label{graded radical}
    Suppose that $\mu\in\P$. Then the induced map
    $$f\map{C^\mu\otimes_R C^\mu}R; x\otimes y\mapsto \<x,y\>_\mu$$  
    is a homogeneous map of degree zero. In particular,
    $$\rad C^\mu=\set{x\in C^\mu|\<x,y\>_\mu=0\text{ for all }y\in C^\mu}.$$
    is a graded submodule of $C^\mu$.
\end{Lemma}

\begin{proof} By Lemma~\ref{symmetric form}, $\rad C^\mu$ is a
    submodule of $C^\mu$ since $\<\ ,\ \>_\mu$ is associative (with
    respect to the anti-automorphism~$*$). It remains to show that the
    bilinear form defines a homogeneous map of degree zero. Suppose
    that $f(x\otimes y)\ne0$, for some $x,y\in C^\mu$. Write
    $x=\sum_i x_i$ and $y=\sum_j y_j$, where $x_i$ and $y_i$ are both
    homogeneous of degree $i$. Then $\<x_i,y_j\>_\mu\ne0$ for some $i$ and
    $j$. Now write $x_i=\sum_\s a_\s c^\mu_\s$ and $y_j=\sum_\t
    b_\t c^\mu_\t$, for $a_\s,b_\t\in R$ such that $a_\s\ne0$ only
    if $\deg\s=i$ and $b_\t\ne0$ only if $\deg\t=j$. Fix any
    $\v\in T(\mu)$. Then by Lemma~\ref{symmetric form},
    $$\<x_i,y_j\>_\mu c^\mu_{\v\v}
           =\sum_{\s,\t} a_\s b_\t
	   \<c^\mu_\s,c^\mu_\t\>_\mu c^\mu_{\v\v}
	   \equiv \sum_{\s,\t}a_\s b_\t c^\mu_{\v\s}c^\mu_{\t\v}
       \pmod{A^{\gdom\mu}}.$$
    Taking degrees of both sides shows that $\<x_i,y_j\>_\mu\ne0$ only if
    $i+j=0$. That is, $\<x,y\>_\mu\ne0$ only if $\deg(x\otimes y)=0$ as
    we wanted to show. Finally, $\rad C^\mu$ is a graded submodule of
    $C^\mu$ because if $x=\sum_i x_i\in\rad C^\mu$ then $x_i\in\rad
    C^\mu$, for all $i$, since $\<\ ,\ \>_\mu$ is homogeneous.
\end{proof}

The Lemma allows us to define a graded quotient of $C^\mu$, for $\mu\in\P$.

\begin{Defn}
    Suppose that $\mu\in\P$. Let $D^\mu=C^\mu/\rad C^\mu$.
\end{Defn}

By definition, $D^\mu$ is a graded right $A$-module. Henceforth, let
$R=K$ be a field and $A=\bigoplus_{z\in\Z}A_z$ a graded cellular
$K$-algebra. Exactly as in the ungraded case (see \cite[Prop.~2.6]{GL}
or \cite[Prop.~2.11-2.12]{M:Ulect}), we obtain the following.

\begin{Lemma}\label{absolutely irreducible}
    Suppose that $K$ is a field and that $D^\mu\ne0$, for $\mu\in\P$. Then:
    \begin{enumerate}
    \item The right $A$-module $D^\mu$ is an absolutely irreducible
        graded $A$-module.
    \item The (graded) Jacobson radical of $C^\mu$ is $\rad C^\mu$.
    \item If $\lambda\in\P$ and $M$ is a graded $A$-submodule of
        $C^\lambda$. Then 
	$$\ZHom_A(C^\mu,C^\lambda/M)\ne0$$ 
	only if $\lambda\gedom\mu$. Moreover, if $\lambda=\mu$ then
        $$\ZHom_A(C^\mu,C^\mu/M)=\Hom_A(C^\mu,C^\mu/M)\cong K.$$
    \end{enumerate}
\end{Lemma}

In particular, if $M$ is a graded $A$-submodule of $C^\mu$ then every
non-zero homomorphism from $C^\mu$ to $C^\mu/M$ is degree preserving.

Let $\P_0=\set{\lambda\in\P|D^{\lambda}\ne0}$. Recall that if $M$ is an
$A$-module then $\underline M$ is the ungraded $\underline A$-module
obtained by forgetting the grading.

\begin{Theorem}\label{graded simples}
    Suppose that $K$ is a field and that $A$ is a graded cellular
    $K$-algebra.
    \begin{enumerate}
        \item If $\mu\in\P_0$ then $D^\mu$ is
        an absolutely irreducible graded $A$-module.
        \item Suppose that $\lambda,\mu\in\P_0$. Then
        $D^\lambda\cong D^\mu\<k\>$, for some $k\in\Z$, if and only
      if~$\lambda=\mu$ and $k=0$.
      \item $\set{D^\mu\<k\>|\mu\in\P_0\text{ and }k\in\Z}$ is a complete set
        of pairwise non-isomorphic graded simple $A$-modules.
  \end{enumerate}
\end{Theorem}

\begin{proof}[Sketch of proof]
    Parts~(a) and~(b) follow directly from Lemma~\ref{absolutely
    irreducible}. For part~(c), observe that, up to degree shift,
    every graded simple $A$-module is isomorphic to a quotient of $A$
    by a maximal graded right ideal. The graded cellular basis of~$A$
    induces a graded filtration of~$A$ with all quotient modules
    isomorphic to direct sums of shifts of graded cell modules, so it
    is enough to show that every composition factor of $C^\lambda$ is
    isomorphic to $D^\mu\<k\>$, for some $\mu\in\P_0$ and some $k\in\Z$.
    Arguing exactly as in the ungraded case completes the proof; see
    \cite[Theorem~3.4]{GL} or \cite[Theorem~2.16]{M:Ulect}.
\end{proof}

In particular, just as Graham and Lehrer~\cite{GL} proved in the ungraded
case, every field is a splitting field for a graded cellular algebra.

\begin{cor}\label{ungraded simples}
  Suppose that $K$ is a field and $A$ is a graded cellular algebra
  over~$K$. Then $\set{\underline{D}^\mu|\mu\in\P_0}$ is a complete set of
  pairwise non-isomorphic ungraded simple $A$-modules.  
\end{cor}

\begin{proof}
    By Lemma~\ref{graded radical}, for each $\lambda\in\P$ the submodule
    $\rad C^\lambda$ is independent of the grading so the ungraded module
    $\underline{D}^\mu$ is precisely the module constructed by using the
    cellular basis of~$A$ obtained by forgetting the grading. Therefore,
    every (ungraded) simple module is isomorphic to $\underline{D}^\mu$ by
    forgetting the grading in Theorem~\ref{graded simples} (or,
    equivalently, by \cite[Theorem~3.4]{GL}).
\end{proof}

\subsection{Graded decomposition numbers}
Recall that $t$ is an indeterminate over~$\Z$. If $M$ is a graded
$A$-module and $D$ is a graded simple module let $[M:D\<k\>]$ be the
multiplicity of the simple module $D\<k\>$ as a graded composition factor
of $M$, for $k\in\Z$. Similarly, let $[\underline M:\underline D]$ the
multiplicity of $\underline D$ as a composition factor of $\underline M$.

\begin{Defn}[Graded decomposition matrices]
  Suppose that $A$ is a graded cellular algebra over a field. Then the
  \textbf{graded decomposition matrix} of $A$ is the matrix
  $\Dec=\(d_{\lambda\mu}(t)\)$, where
  $$d_{\lambda\mu}(t)=\sum_{k\in\Z} [C^\lambda:D^\mu\<k\>]\,t^k,$$
  for $\lambda\in\P$ and $\mu\in\P_0$.
\end{Defn}

Using Lemma~\ref{absolutely irreducible} we obtain the following.

\begin{Lemma}\label{graded decomp}
    Suppose that $\mu\in\P_0$ and $\lambda\in\P$. Then
    \begin{enumerate}
	\item $d_{\lambda\mu}(t)\in\N[t,t^{-1}]$;
  \item $d_{\lambda\mu}(1)=[\underline{C}^\lambda:\underline{D}^\mu]$; and,
	\item $d_{\mu\mu}(t)=1$ and $d_{\lambda\mu}(t)\ne0$ only if $\lambda\gedom\mu$.
    \end{enumerate}
\end{Lemma}

Next we study the graded projective $A$-modules with the aim of
describing the composition factors of these modules using the graded
decomposition matrix.

A graded $A$-module $M$ has a \textbf{graded cell module filtration}
if there exists a filtration
$$0=M_0\subset M_1\subset M_2\subset\dots\subset M_k=M$$ 
such that each $M_i$ is a graded submodule of $M$ and if $1\le i\le k$ then
$M_i/M_{i-1}\cong C^\lambda\<k\>$, for some $\lambda\in\P$ and some
$k\in\Z$. By \cite[Theorem~3.2, Theorem~3.3]{GordonGreen:GradedArtin}, we
know that every projective $A$-module is gradable.

\begin{Prop}\label{cell filtration}
  Suppose that $P$ is a projective $A$ module. Then $P$ has a graded cell
  module filtration.
%In particular, $P$ is a graded $A$-module.
\end{Prop}

\begin{proof}
  Fix a total ordering $\succ$ on
  $\P=\{\lambda_1\succ\lambda_2\succ\dots\succ\lambda_N\}$ which is
  compatible with $\gdom$ in the sense that if $\lambda\gdom\mu$ then
  $\lambda\succ\mu$. Let
  $A(\lambda_i)=\bigcup_{j\le i}A^{\gedom\lambda_i}$. Then
  $$0\subset A(\lambda_1)\subset A(\lambda_2)
                     \subset\dots\subset A(\lambda_N)=A$$
  is a filtration of $A$ by graded two-sided ideals.
  Tensoring with $P$ we have
  $$0\subseteq P\otimes_A A(\lambda_1)\subseteq
    P\otimes_A A(\lambda_2)\subseteq\dots\subseteq P\otimes_AA(\lambda_N)=P,$$
  a graded filtration of $P$. An easy exercise in the definitions
  (cf. \cite[Lemma~2.14]{M:Ulect}), shows that there is a short exact sequence
  $$0\to A(\lambda_{i-1})\to A(\lambda_i)
          \to A^{\gedom\lambda_i}/A^{\gdom\lambda_i}\to0.$$
  Since $P$ is projective, tensoring with $P$ is exact so the subquotients
  in the filtration of $P$ above are
  $$P\otimes_A A(\lambda_i)/P\otimes_AA(\lambda_{i-1})
          \cong P\otimes_A\(A^{\gedom\lambda_i}/A^{\gdom\lambda_i}\)
          \cong P\otimes_A(C^{*\lambda_i}\otimes_RC^{\lambda_i}),$$
   where the last isomorphism comes from (\ref{two cells}). Hence, $P$ has
   a graded cell module filtration as claimed.
\end{proof}

For each  $\mu\in\P_0$ let $P^\mu$ be the projective cover of $D^\mu$.
Then for each $k\in\Z$, $P^\mu\<k\>$ is the projective cover of
$D^\mu\<k\>$.

\begin{Lemma}\label{lm314}
  Suppose that $\lambda\in\P$ and $\mu\in\P_0$. Then:
  \begin{enumerate}
    \item $d_{\lambda\mu}(t)=\Dim{\Hom_A^\Z(P^\mu,C^\lambda)}$.
    \item $\Hom_A^\Z(P^\mu,C^\lambda)\cong P^\mu\otimes_A C^{*\lambda}$ as
      $\Z$-graded $K$-modules.
  \end{enumerate}
\end{Lemma}

\begin{proof} Part (a) follows directly from the definition of
    projective covers. Part~(b) follows using essentially the same
    argument as in the ungraded case; see the proof of \cite[Theorem~3.7(ii)]{GL}.
\end{proof}

\begin{Defn}[Graded Cartan matrix]
    Suppose that $A$ is a graded cellular algebra over a field. Then the
    \textbf{graded Cartan matrix} of $A$ is the matrix
  $\Cart=\(c_{\lambda\mu}(t)\)$, where
  $$c_{\lambda\mu}(t)=\sum_{k\in\Z} [P^\lambda:D^\mu\<k\>]\,t^k,$$
  for $\lambda,\mu\in\P_0$.
\end{Defn}

If $M=(m_{ij})$ is a matrix let $M^{\text{tr}}=(m_{ji})$ be its transpose.

\begin{Theorem}[Graded Brauer-Humphreys reciprocity]\label{cde-triangle}
  Suppose that $K$ is a field and that $A$ is a graded cellular
  $K$-algebra. Then $\Cart = \Dec^{\text{tr}}\Dec$.
\end{Theorem}

\begin{proof}Suppose that $\lambda,\mu\in\P_0$. Then by
  Proposition~\ref{cell filtration} and (\ref{two cells}) we have
  \begin{align*}
    c_{\lambda\mu}(t)&=\sum_{k\in\Z}[P^\lambda:D^\mu\<k\>]\,t^k\\
    &=\sum_{k\in\Z}\sum_{\nu\in\P}
           [(P^\lambda\otimes_A C^{*\nu})\otimes_RC^\nu:D^\mu\<k\>]\,t^k\\
    &=\sum_{k\in\Z}\sum_{\nu\in\P}
         \Dim{P^\lambda\otimes_A C^{*\nu}}[C^\nu:D^\mu\<k\>]\,t^k\\
    &=\sum_{\nu\in\P}\Dim{P^\lambda\otimes_A C^{*\nu}}
         \sum_{k\in\Z}[C^\nu:D^\mu\<k\>]\,t^k\\
         &=\sum_{\nu\in\P}d_{\nu\lambda}(t)d_{\nu\mu}(t),
  \end{align*}
where we have used Lemma \ref{lm314} in the last step.
\end{proof}

Let $K_0(A)$ be the (enriched) Grothendieck group of $A$. Thus,
$K_0(A)$ is the free $\Z[t,t^{-1}]$-module generated by symbols
$[M]$, where $M$ runs over the finite dimensional graded
$A$-modules, with relations $[M\<k\>]=t^k[M]$, for $k\in\Z$, and
$[M]=[N]+[P]$ whenever $0\rightarrow N\to M\to P\rightarrow 0$ is a short exact sequence of
graded $A$-modules. Then $K_0(A)$ is a free $\Z[t,t^{-1]}]$-module
with distinguished bases $\set{[D^\mu]|\mu\in\P_0}$ and
$\set{[C^\mu]|\mu\in\P_0}$. Similarly, let $K_0^*(A)$ be the
(enriched) Grothendieck group of finitely generated (graded)
projective $A$-modules. Then $K_0^*(A)$ is free as a
$\Z[t,t^{-1}]$-module with basis $\set{[P^\mu]|\mu\in\P_0)}$.
Replacing~$\P_0$ with~$\P$ in the definition of $K_0(A)$, gives the
free $\Z[t,t^{-1}]$-module ${\mathscr F(A)}$ which is generated by symbols
$\llbracket C^{\mu}\rrbracket$ for $\mu\in\P$. Theorem \ref{cde-triangle} then says that the
following diagram commutes:
$$\begin{diagram}
  \node{K_0^*(A)} \arrow{e,t}{\Dec}\arrow{se,b}{\Cart}
  \node{\mathscr F(A)} \arrow{s,r}{\Dec^{\text{tr}}}\\
  \node[2]{K_0(A)}
\end{diagram}$$

Recall from Definition~\ref{graded cellular def} that $A$ is equipped
with a graded anti-automorphism~$*$.  Let $M$ be a graded $A$-module.
The \textbf{contragredient dual} of $M$ is the graded $A$-module
$$M^\circledast = \ZHom_A(M,K)=\bigoplus_{d\in\Z}\Hom_A(M\<d\>,K)$$
where the action of $A$ is given by $(fa)(m)=f(ma^*)$, for all $f\in
M^\circledast$, $a\in A$ and~$m\in M$. As a vector space,
$M^\circledast_d=\Hom_A(M_{-d},K)$, so
$\Dim{M^\circledast}=\Dim[t^{-1}]M$.

\begin{Prop} \label{self dual}
  Suppose that $\mu\in\P_0$. Then $D^\mu\cong(D^\mu)^\circledast$.
\end{Prop}

\begin{proof} 
    By Lemma~\ref{graded radical} $\<\ ,\ \>_\mu$ restricts to give a
non-degenerate homogeneous bilinear form of degree zero on $D^\mu$.
Therefore, if $d$ is any non-zero element of $D^\mu$ then the map
$D^\mu\longrightarrow(D^\mu)^\circledast$ given by $d\mapsto\<d,-\>_\mu$
gives the desired isomorphism.
\end{proof}

If $M$ is a graded $A$-module then 
$(M\langle k\rangle)^\circledast\cong (M^\circledast)\langle -k\rangle$ 
as $K$-vector spaces, for any $k\in\Z$. Consequently, contragredient
duality induces a $\Z$-linear automorphism
$\rule[1.2ex]{.6em}{.1ex}\map{K_0(A)}K_0(A)$ which is determined by
$$\overline{ t^k[M^\circledast]} = t^{-k}[M],$$
for all $M\in A\Mod$ and all $k\in\Z$. 

If $\mu\in\P_0$ then $\overline{[D^\mu]}=[D^\mu]$ by
Proposition~\ref{self dual}. Define polynomials
$e_{\lambda\mu}(t)\in\Z[t,t^{-1}]$ by setting
$(e_{\lambda\mu}(-t))=\Dec^{-1}$. Then $e_{\mu\mu}=1$ and
$$[D^\mu]=[C^\mu] +\sum_{\substack{\nu\in\P_0\\\mu\gdom\nu}} 
                  e_{\mu\nu}(-t)[C^\nu].$$
(Following the philosophy of the Kazhdan-Lusztig conjectures, we define the
polynomials $e_{\lambda\mu}(-t)$ in the hope that
$e_{\lambda\mu}(t)\in\N[t]$.) \textit{A priori},
$d_{\lambda\mu}(t)\in\N[t,t^{-1}]$ and $e_{\lambda\mu}(t)\in\Z[t,t^{-1}]$.
In contrast, we have a `Kazhdan-Lusztig basis' for~$K_0(A)$.

\begin{Prop}\label{KL}
   There exists a unique basis $\set{[E^\mu]|\mu\in\P_0}$ of~$K_0(A)$
   such that if $\mu\in\P_0$ then $\overline{[E^\mu]}=[E^\mu]$ and
    $$ [E^\mu]=[C^{\mu}]+\sum_{\substack{\lam\in\P_0\\\mu\gdom\lam}}
                    f_{\mu\lam}(-t)[C^{\lam}],
    $$
    for some polynomials $f_{\mu\lam}(t)\in t\Z[t]$, for $\lambda\in\P_0$.
\end{Prop}

\begin{proof} Using Proposition \ref{self dual} it is easy to see
    that if $\lam\in\P_0$ then there exist polynomials
    $r_{\lam\mu}(t)\in\Z[t,t^{-1}]$, for $\mu\in\P_0$, such that
    $$\overline{[C^{\lam}]}=[C^{\lam}]
           +\sum_{\substack{\mu\in\P_0\\\lambda\gdom\mu}}r_{\lam\mu}(t)[C^{\mu}].
    $$
    The Corollary follows from this observation using a well-known
    inductive argument due to Kazhdan and Lusztig; see
    \cite[Theorem~1.1]{KL} or \cite[1.2]{Du:abstractKL}. 
\end{proof}

It seems unlikely to us that there is a mild condition on $A$ which ensures
that $[E^\mu]=[D^{\mu}]$, or equivalently, $d_{\lambda,\mu}(t)\in t\N[t]$
when $\lambda\gdom\mu$. We conclude this section by discussing a strong
assumption on $A$ which achieves this.

A graded $A$-module $M=\bigoplus_i M_i$ is \textbf{positively graded} if
$M_i=0$ whenever $i<0$.  It is easy to check that a graded cellular
algebra $A$ is positively graded if and only if $\deg\s\ge0$, for all
$\s\in T(\lambda)$, for $\lambda\in\P$. Consequently, if $A$ is positively
graded then so is each cell module of~$A$.

A graded $A$-module $M=\bigoplus_i M_i$ is \textbf{pure of degree $d$}
if $M=M_d$.

\begin{Lemma}
    Suppose that $A$ is a positively graded cellular algebra over a
    field $K$ and suppose that $\lambda\in\P$ and $\mu\in\P_0$. Then:
    \begin{enumerate}
	\item $D^\mu$ is pure of degree $0$; and,
	\item $d_{\lambda\mu}(t)\in\N[t]$.
    \end{enumerate}
\end{Lemma}

\begin{proof}The bilinear form $\<\ ,\ \>$ on $C^\mu$ is homogeneous
    of degree~$0$ by Lemma~\ref{graded radical}. Therefore, if $x,y\in
    C^\mu$ and $\<x,y\>_\mu\ne0$ then $\deg x+\deg y=0$, so that
    $x,y\in C^\mu_0$. This implies~(a). In
    turn, this implies (b) because $D^\mu\<k\>$ can only be a
    composition factor of $C^\lambda$ if  $k\ge0$ (and
    $\lambda\gedom\mu$) since $A$ is positively graded. 
\end{proof}

In the ungraded case, Graham and Lehrer~\cite[Remark 3.10]{GL} observed
that a cellular algebra is quasi-hereditary if and only if $\P=\P_0$. This
is still true in the graded setting. Conversely, any graded split
quasi-hereditary algebra that has a graded duality which fixes the simple
modules is a graded cellular algebra by the arguments of Du and
Rui~\cite[Cor.~6.2.2]{DuRui:strat}.  Similarly, it is easy to see that if
$A$ is a positively graded cellular algebra such that $\P=\P_0$ then
$A\Mod$ is a positively graded highest weight category with duality as
defined in \cite{CPS:HomDual}.

If $M=\bigoplus_{i\ge0}M_i$ is a positive graded $A$-module let
$M_+=\bigoplus_{i>0}M_i$. If $A$ is positively graded then $M_+$ is a
graded $A$-submodule of $M$. Let $\Rad M$ be the Jacobson radical of~$M$.

As the following Lemma indicates, there do exist positively graded
quasi-hereditary cellular algebras such that, in the notation of
Proposition~\ref{KL}, $[D^\mu]\ne[E^\mu]$ for all $\mu\in\P=\P_0$.

\begin{Lemma}
    Suppose that $A$ is a positive graded quasi-hereditary cellular algebra
    over a field. Then the following are equivalent:
    \begin{enumerate}
      \item $A_0\cong A/A_+$ is a (split) semisimple algebra;
      \item $\Rad A=A_+$;
      \item $\rad C^\mu=C^\mu_+$, for all $\mu\in\P$;
      \item $[D^\mu]=[E^\mu]$, for all $\mu\in\P$; and,
      \item $d_{\lambda\mu}(t)\in t\N[t]$, for all $\lambda\ne\mu\in\P$.
    \end{enumerate}
\end{Lemma}

\begin{proof}
    As $A$ is quasi-hereditary, if $\mu\in\P$ then $D^\mu\ne0$ and
    $\rad C^\mu=\Rad C^\mu$ by the general theory of cellular algebras
    (by Lemma~\ref{absolutely irreducible}). Therefore, since $A$ is
    positively graded, all of the statements in the Lemma are easily
    seen to be equivalent to the condition that $D^\mu\cong
    C^\mu/C^\mu_+$, for all $\mu\in\P$.
\end{proof}
%%%%%%%%%%%%%%%%%%%%%%%%%%%%%%%%%%%%%%%%%%%%%%%%%%%%%%%%
\section{Khovanov-Lauda--Rouquier algebras and Hecke algebras}
In this section, following~\cite{BK:GradedKL}, we set our notation and
define the cyclotomic Khovanov-Lauda--Rouquier algebras of type~$A$ and recall
Brundan and Kleshchev's graded isomorphism theorem.

\subsection{Cyclotomic Khovanov-Lauda--Rouquier algebras}
As in section~2, let~$R$ be a commutative integral domain with~$1$.

Throughout this paper we fix an integer $e$ such that either $e=0$ or
$e\ge2$. Let $\Gamma_e$ be the oriented quiver with vertex set $I=\Z/e\Z$
and with directed edges $i\longrightarrow i+1$, for all $i\in I$. Thus,
$\Gamma_e$ is the quiver of type $A_\infty$ if $e=0$, and if $e\ge2$ then
it is a cyclic quiver of type $A^{(1)}_e$:
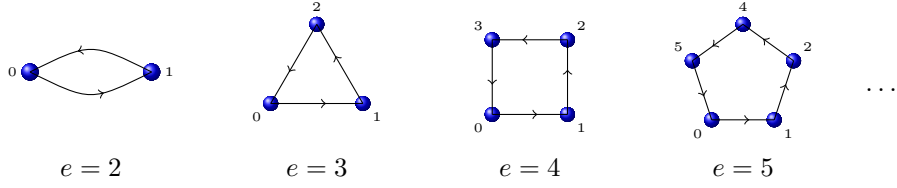
\begin{center}
% e=2
\begin{tabular}{*5c}
  \begin{tikzpicture}[scale=0.8,decoration={curveto, markings,
            mark=at position 0.6 with {\arrow{>}}
    }]
    \useasboundingbox (-1.7,-0.7) rectangle (1.7,0.7);
    \foreach \x in {0, 180} {
      \shade[ball color=blue] (\x:1cm) circle(4pt);
    }
    \tikzstyle{every node}=[font=\tiny]
    \draw[postaction={decorate}] (0:1cm) .. controls (90:5mm) .. (180:1cm) 
           node[left,xshift=-0.5mm]{$0$};
    \draw[postaction={decorate}] (180:1cm) .. controls (270:5mm) .. (0:1cm) 
           node[right,xshift=0.5mm]{$1$};
  \end{tikzpicture} 
& % e=3
  \begin{tikzpicture}[scale=0.7,decoration={ markings,% switch on markings 
            mark=at position 0.6 with {\arrow{>}}
    }]
    \useasboundingbox (-1.7,-0.7) rectangle (1.7,1.4);
    \foreach \x in {90,210,330} {
      \shade[ball color=blue] (\x:1cm) circle(4pt);
    }
    \tikzstyle{every node}=[font=\tiny]
    \draw[postaction={decorate}]( 90:1cm)--(210:1cm) node[below left]{$0$};
    \draw[postaction={decorate}](210:1cm)--(330:1cm) node[below right]{$1$};
    \draw[postaction={decorate}](330:1cm)--( 90:1cm) 
           node[above,yshift=.5mm]{$2$};
  \end{tikzpicture} 
& % e=4
  \begin{tikzpicture}[scale=0.7,decoration={ markings,% switch on markings 
            mark=at position 0.6 with {\arrow{>}}
    }]
    \useasboundingbox (-1.7,-0.7) rectangle (1.7,0.7);
    \foreach \x in {45, 135, 225, 315} {
      \shade[ball color=blue] (\x:1cm) circle(4pt);
    }
    \tikzstyle{every node}=[font=\tiny]
    \draw[postaction={decorate}] (135:1cm) -- (225:1cm) node[below left]{$0$};
    \draw[postaction={decorate}] (225:1cm) -- (315:1cm) node[below right]{$1$};
    \draw[postaction={decorate}] (315:1cm) -- (45:1cm) node[above right]{$2$};
    \draw[postaction={decorate}] (45:1cm) -- (135:1cm) node[above left]{$3$};
  \end{tikzpicture} 
& % e=5
  \begin{tikzpicture}[scale=0.7,decoration={ markings,% switch on markings 
            mark=at position 0.6 with {\arrow{>}}
    }]
    \useasboundingbox (-1.7,-0.7) rectangle (1.7,0.7);
    \foreach \x in {18,90,162,234,306} {
      \shade[ball color=blue] (\x:1cm) circle(4pt);
    }
    \tikzstyle{every node}=[font=\tiny]
    \draw[postaction={decorate}] (162:1cm) -- (234:1cm)node[below left]{$0$};
    \draw[postaction={decorate}] (234:1cm) -- (306:1cm) node[below right]{$1$};
    \draw[postaction={decorate}] (306:1cm) -- (18:1cm) node[above right]{$2$};
    \draw[postaction={decorate}] (18:1cm) -- (90:1cm) 
           node[above,yshift=.5mm]{$4$};
    \draw[postaction={decorate}] (90:1cm) -- (162:1cm) node[above left]{$5$};
  \end{tikzpicture} 
&\raisebox{3mm}{$\dots$}
\\[4mm]
  $e=2$&$e=3$&$e=4$&$e=5$&
\end{tabular}
\end{center}
Let $(a_{i,j})_{i,j\in I}$ be the symmetric Cartan matrix associated
with $\Gamma_e$, so that
$$a_{i,j}=\begin{cases}
    2 & \text{ if }i=j,\\
    0 & \text{ if }i\ne j\pm 1,\\
    -1 & \text{ if $e\ne2$ and }i=j\pm 1,\\
    -2 & \text{ if $e=2$ and }i=j+1.
\end{cases}$$
Following Kac~\cite[Chapt.~1]{Kac}, let $(\mathfrak h,\Pi,\check\Pi)$
be a realization of the Cartan matrix, and $\set{\alpha_i|i\in I}$
the associated set of simple roots, $\set{\Lambda_i|i\in I}$ the
fundamental dominant weights, and $(\cdot,\cdot)$ the bilinear form
determined by
$$(\alpha_i,\alpha_j)=a_{i,j}\qquad\text{and}\qquad
(\Lambda_i,\alpha_j)=\delta_{ij},\qquad\text{for }i,j\in I.$$
Finally, let $P_+=\bigoplus_{i\in I}\N\Lambda_i$ be the dominant
weight lattice of $(\mathfrak h,\Pi,\check\Pi)$ and let
$Q_+=\bigoplus_{i\in I}\N\alpha_i$ be the positive root lattice. The
Kac-Moody Lie algebra corresponding to this data is $\widehat{\mathfrak{sl}}_e$
if $e>0$ and $\mathfrak{sl}_\infty$ if $e=0$.

For the remainder of this paper fix a (dominant weight $\Lambda\in P_+$ and
a non-negative integer~$n$.  Set $\ell=\sum_{i\in I}(\Lambda,\alpha_i)$. A
\textbf{multicharge} for $\Lambda$ is any sequence of integers
$\charge=(\kappa_1,\dots,\kappa_\ell)\in\Z^\ell$ such that 
\begin{enumerate}
  \item $(\Lambda,\alpha_i)=\#\set{1\le s\le\ell|\kappa_s\equiv i\pmod e}$,
for $i\in I$, 
\item if $e\ne0$ then $\kappa_s-\kappa_{s+1}\ge n$, for $1\le s<\ell$, 
\end{enumerate}
where in~(a) we use the convention that $i\pmod e=i$ if $e=0$. 

There are many different choices of multicharge for $\Lambda$. For the rest
of this paper we fix an arbitrary multicharge $\charge$ satisfying the two
conditions above.  For the rest of this paper we fix an arbitrary
multicharge $\charge$ satisfying the two conditions above. All of the
bases considered in this paper, but none of the algebras, depend upon our
choice of multicharge. The assumption that $\kappa_s-\kappa_{s+1}\ge n$
when $e\ne0$ is not essential. It is used in section~4 to streamline our
choice of modular system for the cyclotomic Hecke algebras.

The following algebra has its origins in the work of Khovanov and
Lauda~\cite{KhovLaud:diagI},  Rouquier~\cite{Rouq:2KM} and Brundan and
Kleshchev~\cite{BK:GradedKL}.

\begin{Defn}\label{relations}
  The \textbf{Khovanov-Lauda--Rouquier algebra}, or 
  \textbf{quiver Hecke algebra}, $\R$ of weight
  $\Lambda$ and type $\Gamma_e$ is the unital associative $R$-algebra
  with generators
  $$\{\psi_1,\dots,\psi_{n-1}\} \cup
         \{ y_1,\dots,y_n \} \cup \set{e(\bi)|\bi\in I^n}$$
  and relations
\begin{align*}
y_1^{(\Lambda,\alpha_{i_1})}e(\bi)&=0,
& e(\bi) e(\bj) &= \delta_{\bi\bj} e(\bi),
&{\textstyle\sum_{\bi \in I^n}} e(\bi)&= 1,\\
y_r e(\bi) &= e(\bi) y_r,
&\psi_r e(\bi)&= e(s_r{\cdot}\bi) \psi_r,
&y_r y_s &= y_s y_r,\\
\end{align*}
\vskip-38pt
\begin{align*}
\psi_r y_s  &= y_s \psi_r,&\text{if }s \neq r,r+1,\\
\psi_r \psi_s &= \psi_s \psi_r,&\text{if }|r-s|>1,\\
\end{align*}
\vskip-30pt
\begin{align*}
  \psi_r y_{r+1} e(\bi) &= \begin{cases}
      (y_r\psi_r+1)e(\bi),\hspace*{18mm} &\text{if $i_r=i_{r+1}$},\\
    y_r\psi_r e(\bi),&\text{if $i_r\neq i_{r+1}$}
  \end{cases} \\
  y_{r+1} \psi_re(\bi) &= \begin{cases}
      (\psi_r y_r+1) e(\bi),\hspace*{18mm} &\text{if $i_r=i_{r+1}$},\\
    \psi_r y_r e(\bi), &\text{if $i_r\neq i_{r+1}$}
  \end{cases}\\
  \psi_r^2e(\bi) &= \begin{cases}
       0,&\text{if $i_r = i_{r+1}$},\\
      e(\bi),&\text{if $i_r \ne i_{r+1}\pm1$},\\
      (y_{r+1}-y_r)e(\bi),&\text{if  $e\ne2$ and $i_{r+1}=i_r+1$},\\
       (y_r - y_{r+1})e(\bi),&\text{if $e\ne2$ and  $i_{r+1}=i_r-1$},\\
      (y_{r+1} - y_{r})(y_{r}-y_{r+1}) e(\bi),&\text{if $e=2$ and $i_{r+1}=i_r+1$}
\end{cases}\\
% \end{align*}
% \vskip-36pt
%\begin{align*}
\psi_{r}\psi_{r+1} \psi_{r} e(\bi) &= \begin{cases}
    (\psi_{r+1} \psi_{r} \psi_{r+1} +1)e(\bi),\hspace*{7mm}
       &\text{if $e\ne2$ and $i_{r+2}=i_r=i_{r+1}-1$},\\
  (\psi_{r+1} \psi_{r} \psi_{r+1} -1)e(\bi),
       &\text{if $e\ne2$ and $i_{r+2}=i_r=i_{r+1}+1$},\\
  \big(\psi_{r+1} \psi_{r} \psi_{r+1} +y_r\\
  \qquad -2y_{r+1}+y_{r+2}\big)e(\bi),
    &\text{if $e=2$ and $i_{r+2}=i_r=i_{r+1}+1$},\\
  \psi_{r+1} \psi_{r} \psi_{r+1} e(\bi),&\text{otherwise.}
\end{cases}
\end{align*}
for $\bi,\bj\in I^n$ and all admissible $r, s$.
\end{Defn}

It is straightforward, albeit slightly tedious, to check that all of these
relations are homogeneous with respect to the following degree function
on the generators
$$\deg e(\bi)=0,\qquad \deg y_r=2\qquad\text{and}\qquad \deg
  \psi_s e(\bi)=-a_{i_s,i_{s+1}},$$
for $1\le r\le n$, $1\le s<n$ and $\bi\in I^n$.
Therefore, the Khovanov-Lauda--Rouquier algebra $\R$ is $\Z$-graded. From
this presentation, however, it is not clear how to construct a basis for
$\R$, or even what the dimension of $\R$ is.

\subsection{Cyclotomic Hecke algebras}
Throughout this section we fix an invertible element $q\in R$. Let
$\delta_{q1}=1$ if $q=1$ and set $\delta_{q1}=0$ otherwise.

\begin{Defn}
    Suppose that $q\in R$ is an invertible element of $R$ and that
    $\bQ=(Q_1,\dots,Q_\ell)\in R^\ell$. The \textbf{cyclotomic Hecke
    algebra} $\HH_n(q,\bQ)=\HH_n^R(q,\bQ)$ of type $G(\ell,1,n)$ and with
    parameters $q$ and $\bQ$ is the unital associative $R$-algebra
    with generators $L_1,\dots,L_n,T_1,\dots,T_{n-1}$ and relations
    \begin{align*}
    (L_1-Q_1)\dots(L_1-Q_\ell)&=0,
    & L_rL_s&=L_sL_r,\\
     (T_r+1)(T_r-q)&=0,
    &T_rL_r+\delta_{q1}&=L_{r+1}(T_r-q+1),\\
    T_sT_{s+1}T_s&=T_{s+1}T_sT_{s+1},\\
    T_rL_s&=L_sT_r,&\text{if }s\ne r,r+1,\\
    T_rT_s&=T_sT_r,&\text{if }|r-s|>1,
  \end{align*}
  where $1\le r<n$ and $1\le s<n-1$.
\end{Defn}

\begin{Remark}
  If $q\ne1$ then it is straightforward using \cite[Lemma~3.3]{AK} to show
  that the algebra $\HH_n(q,\bQ)$ is isomorphic to the Hecke algebra of
  type $G(\ell,1,n)$ with parameters $q$ and $\bQ$. If $q=1$ then the
  relations above reduce to the relations for the degenerate Hecke algebra
  of type $G(\ell,1,n)$ with parameters $\bQ$; see, for
  example,~\cite[Chapt.~3]{Klesh:book}. By giving a uniform presentation
  for the degenerate and non-degenerate Hecke algebras we can emphasize
  where it is important whether or not $q=1$ in what follows.
\end{Remark}

Let $\Sym_n$ be the symmetric group of degree $n$ and let
$s_i=(i,i+1)\in\Sym_n$, for $1\le i<n$. Then $\{s_1,\dots,s_{n-1}\}$
is the standard set of Coxeter generators for $\Sym_n$. If
$w\in\Sym_n$ then the \textbf{length} of $w$ is
$$\ell(w)=\min\set{k|w=s_{i_1}\dots s_{i_k} \text{ for some }1\le i_1,\dots,i_k<n}.$$
If $w=s_{i_1}\dots s_{i_k}$ with $k=\ell(w)$ then $s_{i_1}\dots s_{i_k}$ is
a \textbf{reduced expression} for~$w$. In this case, set 
$T_w:=T_{i_1}\dots T_{i_k}$. Then $T_w$ is independent of the choice of
reduced expression because the generators $T_1,\dots,T_{n-1}$ satisfy the
braid relations of $\Sym_n$; see, for example, \cite[Theorem~1.8]{M:Ulect}.
Note that $L_{i+1}=q^{-1}T_iL_iT_i+\delta_{q1}T_i$, 
for $i=1,\dots,n-1$. By
\cite[Theorem~3.10]{AK} and \cite[Theorem~7.5.6]{Klesh:book},
$$\set{L_1^{a_1}\dots L_n^{a_n}T_w|0\le a_1,\dots,a_n<\ell\text{ and }
         w\in\Sym_n}$$
is an $R$-basis of $\HH_n(q,\bQ)$. 

In order to make the connection with the KLR algebras define the
\textit{quantum characteristic} of $q\in K$ to be the integer $e$ which is
minimal such that $1+q+\dots+q^{e-1}=0$, and where we set $e=0$ if no such $e$
exists. Recall from the last subsection that we have fixed a quiver
$\Gamma_e$, a dominant weight $\Lambda\in P_+$ and a multicharge 
$\charge=(\kappa_1,\dots,\kappa_\ell)$.
Define $\bQ_\Lambda=(q_{\kappa_1},\dots,q_{\kappa_\ell})$, where for an
integer $k\in\Z$ we set
$$q_k=\begin{cases}q^k,&\text{if }q\ne1,\\
                     k,&\text{if }q=1.
\end{cases}$$
If $R=K$ is a field then $\bQ_\Lambda$ depends only on $\Lambda$
and not on the choice of multicharge~$\charge$.

\begin{Defn}
  Suppose that $R=K$ is a field of characteristic $p\ge0$ and $q$ is a
  non-zero element of $K$. Let $e$ be the quantum characteristic of $q$ and
  $\Lambda\in P_+$ a dominant weight for $\Gamma_e$. Then the
  cyclotomic Hecke algebra of \textbf{weight} $\Lambda$ is the 
  algebra $\H=\HH_n(q,\bQ_\Lambda)$. 
\end{Defn}

Recall from the subsection~\Sect3.1 that $I=\Z/e\Z$. If $i\in I$ then we
set $q_i=q_\iota$, where $\iota\in\Z$ and $i\equiv\iota\pmod e$. Then $q_i$
is well-defined since $e$ is the quantum characteristic of~$q$.

Suppose that $M$ is a finite dimensional $\H$-module. Then, by
\cite[Lemma~4.7]{Groj:control} and \cite[Lemma~7.1.2]{Klesh:book}, the
eigenvalues of each $L_m$ on $M$ are of the form $q_i$ for $i\in I$. So $M$
decomposes as a direct sum $M=\bigoplus_{\bi\in I^n}M_{\bi}$ of its
generalized eigenspaces, where
$$ M_{\bi}:=\set{v\in M|v(L_r-q_{i_r})^{k}=0
                \text{ for $r=1,2,\cdots,n$ and } k\gg 0}.
$$
(Clearly, we can take $k=\dim M$ here.) In particular, taking $M$ to be the
regular $\H$-module we get a system $\bigl\{e(\bi)\bigm|\bi\in
I^n\bigr\}$ of pairwise orthogonal idempotents in $\H$ such that
$Me(\bi)=M_{\bi}$ for each finite dimensional right $\H$-module $M$.
Note that these idempotents are not, in general, primitive. Moreover, all
but finitely many of the $e(\bi)$'s are zero and, by the relations, their
sum is the identity element of $\R$.

Following Brundan and Kleshchev~\cite[\Sect3,\Sect5]{BK:GradedKL} we now
define elements of $\H$ which satisfy the relations of~$\R$.
For $r=1,\dots,n$ define
$$y_r=\begin{cases}
  \Sum_{\bi\in I^n}(1 - q ^{-i_r}L_r)e(\bi),&\text{if }q\ne1,\\[5mm]
  \Sum_{\bi\in I^n}(L_r - i_r)e(\bi),&\text{if }q=1.
\end{cases}$$
By  \cite[Lemma~2.1]{BK:GradedKL}, or using (\ref{Murphy action}) below, 
$y_1,\dots,y_n$ are nilpotent elements of $\H$, so any power series in
$y_1,\dots,y_n$ can be interpreted as elements of $\H$.  Using this
observation, Brundan and 
Kleshchev~\cite[(3.22),(3,30),(4.27),(4.36)]{BK:GradedKL} define formal 
power series
$P_r(\bi),Q_r(\bi)\in R[y_r,y_{r+1}]$, for $1\le r<n$ and $\bi\in I^n$, and
then set
$$\psi_r=\sum_{\bi\in I^n}(T_r+P_r(\bi))Q_r(\bi)^{-1}e(\bi).$$

Recall that $K$ is a field of characteristic $p\ge0$ and
$e\in\{0,2,3,4,\dots\}$ is the quantum characteristic of $q\in K$. Hence,
we are in one of the following three cases:
\begin{enumerate}
\item $e=p$ and $q=1$;
\item $e=0$ and $q$ is not a root of unity in $K$;
\item $e>1$, $p\nmid e$ and $q$ is a primitive $e^{\text{th}}$ root of unity in $K$.
\end{enumerate}
We are abusing notation here because we are not distinguishing between the
generators of the cyclotomic Khovanov-Lauda--Rouquier algebra and the
elements that we have just defined in $\H$. This abuse is justified by the
Brundan-Kleshchev graded isomorphism theorem.

\begin{Theorem}[\protect{Brundan--Kleshchev~\cite[Theorem~1.1]{BK:GradedKL}}]
    \label{BK main} The map $\R\longrightarrow\H$ which
    sends $$e(\bi)\mapsto e(\bi),\qquad y_r\mapsto
    y_r\qquad\text{and}\qquad \psi_s\mapsto\psi_s,$$ for $\bi\in I^n$,
    $1\le r\le n$ and $1\le s<n$, extends uniquely to an isomorphism
    of algebras. An inverse isomorphism is given by
    $$L_r\mapsto\begin{cases}
        \Sum_{\bi\in I^n}q^{i_r}(1-y_r)e(\bi),&\text{if }q\neq1,\\[4mm]
        \Sum_{\bi\in I^n}(y_r+i_r)e(\bi), &\text{if }q=1,
                \end{cases}$$
    and $T_s\mapsto\Sum_{\bi\in I^n}\(\psi_s Q_s(\bi)-P_s(\bi)\)e(\bi)$,
    for $1\le r\le n$ and $1\le s<n$.
\end{Theorem}

Hereafter, we freely identify the algebras $\R$ and $\H$, and
their generators, using this result. In particular, we consider $\H$
to be a $\Z$-graded algebra. All $\H$-modules will be $\Z$-graded
unless otherwise noted.

\subsection{Tableaux combinatorics and the standard basis} We close this
section by introducing some combinatorics and defining the standard basis
of $\HH_n(q,\bQ)$, where $q\in R$ and $\bQ\in R^\ell$ are arbitrary.

Recall that an \textbf{multipartition}, or $\ell$-partition, of $n$ is
an ordered sequence $\blam=(\lambda^{(1)},\dots,\lambda^{(\ell)})$ of
partitions such that $|\lambda^{(1)}|+\dots+|\lambda^{(\ell)}|=n$.
The partitions $\lambda^{(1)},\dots,\lambda^{(\ell)}$ are the
\textbf{components} of $\blam$.  Let $\Multiparts$ be the set of
multipartitions of $n$. Then $\Multiparts$ is partially ordered by
\textbf{dominance} where $\blam\unrhd\bmu$ if
$$\sum_{t=1}^{s-1}|\lambda^{(t)}|+\sum_{i=1}^j\lambda^{(s)}_i
\ge\sum_{t=1}^{s-1}|\mu^{(t)}|+\sum_{i=1}^j\mu^{(s)}_i$$ for all $1\le
s\le\ell$ and all $j\ge1$. We write $\blam\rhd\bmu$ if
$\blam\unrhd\bmu$ and $\blam\neq\bmu$.

The \textbf{diagram} of an multipartition $\blam\in\Multiparts$ is
the set
$$[\blam]=\set{(r,c,l)|1\le c\le \lambda^{(l)}_r, r\ge0
                          \text{ and }1\le l\le\ell},$$
which we think of as an ordered $\ell$-tuple of the diagrams of the
partitions $\lambda^{(1)},\dots,\lambda^{(\ell)}$. A
\textbf{$\blam$-tableau} is a bijective map
$\t\map{[\blam]}\{1,2,\dots,n\}$. We think of
$\t=(\t^{(1)},\dots,\t^{(\ell)})$ as a labeling of the diagram of $\blam$.
This allows us to talk of the rows, columns and components of $\t$.
If~$\t$ is a $\blam$-tableau then set $\Shape(\t)=\blam$.

A \textbf{standard $\blam$-tableau} is a $\blam$-tableau in which,
in each component, the entries increase along each row and down each
column. Let $\Std(\blam)$ be the set of standard $\blam$-tableaux and set
$\Std(\Multiparts)=\bigcup_{\bmu\in\Multiparts}\Std(\bmu)$.

If $\t$ is a standard $\blam$-tableau let $\rest\t k$ be the
subtableau of $\t$ labeled by $1,\dots,k$ in $\t$. If
$\s\in\Std(\blam)$ and $\t\in\Std(\bmu)$ then $\s$
\textbf{dominates} $\t$, and we write $\s\gedom\t$, if
$\Shape(\rest\s k)\gedom\Shape(\rest\t k)$, for $k=1,\dots,n$.
Again, we write $\s\gdom\t$ if $\s\gedom\t$ and $\s\ne\t$. Extend
the dominance partial ordering to pairs of partitions of the same
shape by declaring that $(\u,\v)\gdom(\s,\t)$, for
$(\s,\t)\in\Std(\blam)^2$ and $(\u,\v)\in\Std(\bmu)^2$, if
$(\s,\t)\ne(\u,\v)$ and either $\bmu\gdom\blam$, or $\bmu=\blam$ and
$\u\gedom\s$ and $\v\gedom\t$.

Let $\tlam$ be the unique standard $\blam$-tableau such that
$\tlam\gedom\t$ for all $\t\in\Std(\blam)$. Then $\tlam$ has the
numbers $1,\dots,n$ entered in order, from left to right and then top
to bottom in each component, along the rows of $\blam$. The symmetric
group acts on the set of $\blam$-tableaux. If $\t\in\Std(\blam)$  let
$d(\t)$ be the permutation in $\Sym_n$ such that $\t=\tlam d(\t)$.

Recall from section 3.1 that we have fixed a multicharge
$\charge=(\kappa_1,\dots,\kappa_\ell)$ which determines $\Lambda$. 

\begin{Defn}[\protect{\!\!\cite[Definition~3.14]{DJM:cyc}}]\label{mst}
    Suppose that $\blam\in\Multiparts$ and $\s,\t\in\Std(\blam)$.
    Define $m_{\s\t}=T_{d(\s)^{-1}}m_\blam T_{d(\t)}$, where
    $$m_\blam=\prod_{s=2}^\ell\prod_{k=1}^{|\lambda^{(1)}|+\dots+|\lambda^{(s-1)}|}
    (L_k-q_{\kappa_s})\cdot\sum_{w\in\Sym_\blam}T_w.$$
\end{Defn}

Here and below whenever an element of $\H$ is indexed by a pair of standard
tableaux then these tableaux will always be assumed to have the same shape.

\begin{Theorem}[\protect{Standard basis theorem
  \cite[Theorem~3.26]{DJM:cyc} and \cite[Theorem~6.3]{AMR}}]
  \label{standard basis}
  The set
  $\set{m_{\s\t}|\s,\t\in\Std(\blam)\text{ for }\blam\in\Multiparts}$
  is a cellular basis of $\H$.
\end{Theorem}

In general the standard basis elements $m_{\s\t}$ are not homogeneous.

Using the theory of (ungraded) cellular algebras from section~2 (or
\cite{GL}), we could now construct Specht modules, or cell modules, for
$\H$.  We postpone doing this until section~5, however, where we are able
to define \textit{graded} Specht modules using Theorem~\ref{psi basis} and
the theory of graded cellular algebras developed in section~2.

Suppose that $\blam\in\Multiparts$ and $\gamma=(r,c,l)\in[\blam]$. The
\textbf{residue} of $\gamma$ is 
\begin{equation}\label{residue}
  \res^R(\gamma)=\begin{cases}
    q^{c-r}Q_l,&\text{if }q\ne1,\\
    c-r+Q_l,&\text{if }q=1.
\end{cases}
\end{equation}
If $\t$ is a standard $\blam$-tableau and $1\le k\le n$ set
$\res^R_\t(k)=\res^R(\gamma)$, where $\gamma$ is the unique node in
$[\blam]$ such that  $\t(\gamma)=k$. We emphasize that $\res(\alpha)$ and
$\res_\t(k)$ both depend very much on the base ring and on the choice of
parameters $q$ and $\bQ$ -- and, in particular, whether or not $q=1$. When
we are working over the field $K$ with parameters $\bQ=\bQ_\Lambda$ then
write $\res(\alpha)=\res^K(\alpha)$ and $\res_\t(k)=\res_\t^K(k)$.

The point of these definitions is that by \cite[Prop.~3.7]{JM:cyc-Schaper}
and \cite[Lemma~6.6]{AMR}, there exist scalars $r_{\u\v}\in K$ such that
\begin{equation}\label{Murphy action}
m_{\s\t}L_k=\res^R_\t(k) m_{\s\t}+
                \sum_{(\u,\v)\gdom(\s,\t)}r_{\u\v}m_{\u\v}.
\end{equation}
If $\t\in\Std(\blam)$ is a standard $\blam$-tableau
then its \textbf{residue sequence} $\res(\t)$ is the sequence
$$\res(\t)=\(\res_\t(1),\dots,\res_\t(n)\).$$
We also write $\bi^\t=\res(\t)$. Set
$\Std(\bi)=\coprod_{\blam\in\Multiparts}\set{\t\in\Std(\blam)|\res(\t)=\bi}$.

Finally, we will need to know when $\HH_n(q,\bQ)$ is semisimple.
\begin{Prop}[\protect{\cite[Main theorem]{Ariki:ss} and
  \cite[Theorem~6.11]{AMR}}]\label{semisimple}
  Suppose that $R=K$ is a field of characteristic $p\ge0$. Then the Hecke 
  $\HH_n(q,\bQ)$ is semisimple if and only if
  either $e=0$ or $e>n$, and $P_\HH(q,\bQ)\ne0$ where
  $$P_\HH(q,\bQ)=\begin{cases}
    \Prod_{1\le r<s\le\ell}\prod_{-n<d<n} (q^dQ_r-Q_s),
             &\text{if }q\ne1,\\[5mm]
    \Prod_{1\le r<s\le\ell}\prod_{-n<d<n} (d+Q_r-Q_s),
             &\text{if }q=1.
  \end{cases}$$
\end{Prop}

%%%%%%%%%%%%%%%%%%%%%%%%%%%%%%%%%%%%%%%%%%%%%%%%%%%%%%%%
\section{The seminormal basis and homogeneous elements of $\H$}
The aim of this section is to give an explicit description of the non-zero
idempotents $e(\bi)$ in terms of certain primitive idempotents for the
algebra $\H$ in the semisimple case. We then use this description to
construct a family of homogeneous elements in $\H$ indexed by~$\Multiparts$.

\subsection{The Khovanov-Lauda--Rouquier idempotents}
Let $\L=\<L_1,\dots,L_n\>$ be the subalgebra of $\H$ generated by the
Jucys-Murphy elements of $\H$. Then $\L$ is a commutative subalgebra of
$\H$.

\begin{Lemma}\label{idempotents}
  Suppose that $e(\bi)\ne0$, for $\bi\in I^n$. Then:
  \begin{enumerate}
    \item $e(\bi)$ is the unique idempotent in $\H$ such that $\HH_\bj
      e(\bi)=\delta_{\bi\bj}\HH_\bi$, for $\bj\in I^n$;
    \item $e(\bi)$ is a primitive idempotent in $\L$; and,
    \item $\bi=\res(\t)$ for some standard tableau $\t$.
  \end{enumerate}
  Thus, the idempotents $\set{e(\bi)|\bi\in I^n}\setminus\{0\}$ are the
  (central) primitive idempotents of $\L$.
\end{Lemma}

\begin{proof}
  By definition, $\HH_\bj e(\bi)=\delta_{\bi\bj}\HH_\bi$ so (a) follows
  since $e(\bi)\in\H e(\bi)$. Next, observe that every irreducible
  representation of $\L$ is one dimensional since $\L$ is a commutative
  algebra over a field. Further, modulo more dominant terms,~$L_k$ acts on
  the standard basis element $m_{\s\t}$ as multiplication by
  $\res_\t(k)$ by (\ref{Murphy action}). Therefore, the standard basis
  of $\H$ induces an $\L$-module filtration of $\H$ and the irreducible
  representations of $\L$ are indexed by the residue sequences $\res(\t)\in
  I^n$, for $\t$ a standard $\blam$-tableau for some $\blam\in\Multiparts$.
  Consequently, the decomposition $\H=\bigoplus\HH_\bi$ is nothing more
  than the decomposition of $\H$ into a direct sum of block components when
  $\H$ is considered as an $\L$-module by restriction. Parts~(b) and~(c) now
  follow.
\end{proof}

The following result indicates the difficulties of working with the
homogeneous presentation of~$\H$: we do not know how to prove this result
without recourse to Brundan and Kleshchev's graded isomorphism $\R\cong\H$
(Theorem~\ref{BK main}).

\begin{cor}
  As (graded) subalgebras of $\H$, $\L=\<y_1,\dots,y_n, e(\bi)\mid\bi\in I^n\>$.
\end{cor}

\begin{proof}By Theorem~\ref{BK main}, if $1\le r\le n$ then $y_r\in\L$ and
  $L_r\in\<y_1,\dots,y_n, e(\bi)\mid\bi\in I^n\>$.  Further, by
  Lemma~\ref{idempotents}, $e(\bi)\in\L$, for $\bi\in I^n$. Combining
  these two observations proves the Corollary.
 % (In fact, $\L=\<y_1,\dots,y_n\>$.)
\end{proof}

\subsection{Idempotents and the seminormal form}
Recall that $\H$ is a $K$-algebra, where $K$ is a field of characteristic
$p\ge0$. Lemma~4.2 of~\cite{M:seminormal} explicitly constructs a family
of idempotents in $\H$ which are indexed by the residue sequences of
standard tableaux. As we now recall, these idempotents are defined by
`modular reduction' from the semisimple case.

To describe this modular reduction process we need to choose a modular
system. Unfortunately, the choice of modular system depends upon the
parameters~$q$ and~$\bQ_\Lambda$. To define $\O$ let $x$ be an indeterminate 
over $K$ and set 
$$\O=\begin{cases} K[x]_{(x)}, &\text{if $q\ne1$ or $e=0$},\\
                   \Z_{(p)},&\text{if $q=1$ and $e>0$}.
     \end{cases}$$
Note that if $q=1$ and $e>1$ then $e=p$, the characteristic of~$K$ and
$\O=\Z_{(p)}$ is the localization of $\Z$ at the prime~$p$. In all of the
other cases $\O$ is the localization of $K[x]$ at $x=0$ (note that
$x+q$ is invertible in $\O$ since $q\ne0$). In both cases,
$\O$ is a discrete valuation ring with maximal ideal~$\m=\pi\O$, where 
$\pi=p$ if $q=1$ and $e>0$, and $\pi=x$ otherwise. Let $\K$ be the field of
fractions of $\O$ and consider $\O$ as a subring of $\K$. The triple
$(\O,\K,K)$ is our modular system. In order to exploit it, however, we need
to make a choice of parameters in $\O$.

\begin{Defn}
  Let $\HO=\HO(v,\bQ_\Lambda^\O)$ and let $\HK=\HO\otimes_\O\K$, where
  $$v=\begin{cases}
        x+q,&\text{if $q\ne1$ and $e>0$},\\
        q,  &\text{otherwise},
      \end{cases},
      \qquad
    Q_r^\O=\begin{cases}
      (x+q)^{\kappa_r}, &\text{if $q\neq 1$ and $e>0$},\\
      x^{r}+q^{\kappa_r},&\text{if $q\ne1$ and $e=0$},\\
      \kappa_r,&\text{if $q=1$ and $e>0$},\\
      rx+\kappa_r,&\text{if $q=1$ and $e=0$},
    \end{cases}$$
    for $1\le r\le\ell$, and $\bQ_\Lambda^\O=(Q_1^\O,\dots,Q_\ell^\O)$.
\end{Defn}

The point of these definitions is that the algebra $\HK$ is (split)
semisimple. This follows easily using the semisimplicity criterion in
Proposition~\ref{semisimple} together with definition of the
multicharge~$\charge$.  Specifically, this is where we use the assumption
that if $e>0$ then $\kappa_r-\kappa_{r+1}\ge n$, for $1\le r<\ell$.

% $e\ne0$ then $\kappa_r-\kappa_{r+1}\ge n$, for $1\le r<\ell$.

Recall the definition of residue $\res^R$ from (\ref{residue}) and suppose
that $\blam\in\Multiparts$. Define the \textbf{content} of the node
$\gamma\in[\blam]$ to be $\cont(\gamma)=\res^\O(\gamma)$. Similarly, if
$\t$ is a standard $\blam$-tableau and $1\le k\le n$ we set
$\cont_\t(k)=\res^\O_\t(k)$. Explicitly, by (\ref{residue}) and the
definitions above, if $\t(\gamma)=k$ where $\gamma=(r,c,l)$ then
$$\cont_\t(k)=\cont(\gamma)=\begin{cases}
  (x+q)^{c-r+\kappa_l},&\text{if $q\ne1$ and $e>0$},\\
  q^{c-r}(x^{l}+q^{\kappa_l}),&\text{if $q\ne1$ and $e=0$},\\
  c-r+\kappa_l,&\text{if $q=1$ and $e>0$},\\
  c-r+lx+\kappa_l,&\text{if $q=1$ and $e=0$}.
\end{cases}$$
Note that $\res_\t(k)=\cont_\t(k)\otimes_\O 1_K$.
By (\ref{Murphy action}) in $\HO$ and $\HK$ we have
$$
m_{\s\t}L_k=\cont_\t(k) m_{\s\t}+
                \sum_{(\u,\v)\gdom(\s,\t)}r_{\u\v}m_{\u\v},
$$
for some scalars $r_{\u\v}$. It follows that  $L_1,\dots,L_n$ is a family
of \textit{JM elements} for $\H$ in the sense of
\cite[Definition~2.4]{M:seminormal}. Hence, we can apply the results
from~\cite{M:seminormal} to the algebras~$\HO$, $\HK$ and~$\H$. In
particular, we have the following definition.

\begin{Defn}[\protect{\cite[Defn~3.1]{M:seminormal}}]\label{seminormal basis}
    Suppose that $\blam\in\Multiparts$ and $\s,\t\in\Std(\blam)$. Define
    $$F_\t=\prod_{k=1}^n\prod_{\substack{\s\in\Std(\Multiparts)\\
             \cont_\s(k)\ne\cont_\t(k)}}
    \frac{L_k-\cont_\s(k)}{\cont_\t(k)-\cont_\s(k)}\in\HK.$$
    Set $f_{\s\t}=F_\s m_{\s\t} F_\t$.
\end{Defn}

By (\ref{Murphy action}),
$f_{\s\t}=m_{\s\t}+\sum_{(\u,\v)\gdom(\s,\t)}r_{\u\v}m_{\u\v}$, for
some $r_{\u\v}\in\K$. Therefore,
$$\set{f_{\s\t}|\s,\t\in\Std(\blam)\text{ for }\blam\in\Multiparts}$$
is a basis of $\HK$. This basis is the \textbf{seminormal basis} of $\HK$;
see \cite[Theorem~3.7]{M:seminormal}.  The next definition, which is the key
to what follows, allows us to write $F_\t$ in terms of the
seminormal basis and hence connect these elements with the graded
representation theory.

Let $\blam$ be a multipartition. The node $\alpha=(r,c,l)\in[\blam]$
is an \textbf{addable node} of~$\blam$ if
$\alpha\notin[\blam]$ and $[\blam]\cup\{\alpha\}$ is the diagram of
a multipartition. Similarly, $\rho\in[\blam]$ is a \textbf{removable
node} of $\blam$ if $[\blam]\setminus\{\rho\}$ is the diagram of a
multipartition. Given two nodes $\alpha=(r,c,l)$ and $\beta=(s,d,m)$
then $\alpha$ is \textbf{below} $\beta$ if either $l>m$, or $l=m$
and $r>s$.

The following definition appears as \cite[(2.8)]{M:gendeg} in the
non-degenerate case and it can easily be proved by induction using
\cite[Lemma~6.10]{AMR} in the non-degenerate case.

\begin{Defn}[\cite{M:gendeg,AMR}]\label{gamma}
  Suppose that $\blam\in\Multiparts$ and $\t\in\Std(\blam)$.  For
  $k=1,\dots,n$ let $\Add_\t(k)$ be the set of addable nodes of the
  multipartition $\Shape(\rest\t k)$ which are \textit{below} $\t^{-1}(k)$.
  Similarly, let $\Rem_\t(k)$ be the set of removable nodes of
  $\Shape(\rest\t k)$ which are \textit{below} $\t^{-1}(k)$. Now define
  $$\gamma_\t=v^{\ell(d(\t))+\delta(\blam)}\prod_{k=1}^n
            \dfrac{\prod_{\alpha\in\Add_\t(k)}
                      \(\cont_\t(k)-\cont(\alpha)\)}%
                  {\prod_{\rho\in\Rem_\t(k)}
                   \(\cont_\t(k)-\cont(\rho)\)}\quad\in\K,
  $$
  where $\delta(\blam)=\frac12\sum_{s=1}^\ell
        \sum_{i\ge1}(\lambda^{(s)}_i-1)\lambda^{(s)}_i$.
\end{Defn}

It is an easy exercise in the definitions to check that the terms in the
denominator of $\gamma_\t$ are never zero so that $\gamma_\t$ is a
well-defined element of $\K$. As the algebra $\HK$ is semisimple we have
the following.

\begin{Lemma}[\protect{\cite[Theorem 3.7]{M:seminormal}}]\label{primitive}
  Suppose that $\blam\in\Multiparts$ and $\t\in\Std(\blam)$. Then
  $F_\t=\frac1{\gamma_\t} f_{\t\t}$
  is a primitive idempotent in $\HK$.
\end{Lemma}

For any standard tableau~$\t$ and an integer $k$, with $1\le k\le n$,
define sets $\LAdd_\t(k)$ and $\LRem_\t(k)$ by
\begin{align*}
\LAdd_\t(k)&=\set{\alpha\in\Add_\t(k)|\res(\alpha)=\res_\t(k)}\\
\text{and}\quad
\LRem_\t(k)&=\set{\rho\in\Rem_\t(k)|\res(\rho)=\res_\t(k)}.
\end{align*}
Using this notation we can give a non-recursive definition of the
Brundan-Kleshchev-Wang degree function on standard tableaux.

\begin{Defn}
   [\protect{Brundan, Kleshchev and Wang~\cite[Defn.~3.5]{BKW:GradedSpecht}}]
   \label{tableau degree}
   Suppose that $\blam\in\Multiparts$ and that $\t$ is a standard
   $\blam$-tableau. Then
   $$
     \deg \t=\sum_{k=1}^n\Big(|\LAdd_\t(k)|-|\LRem_\t(k)|\Big),
   $$
 \end{Defn}

% \begin{Lemma}\label{gamma degree}
%   Suppose that $\K$ is a field of characteristic zero and that
%   $\t\in\Std(\blam)$, for some $\blam\in\Multiparts$.
%   Then $\gamma_\t=u_\t x^{\deg\t}$, for some unit $u_\t$ in $\O$.
% \end{Lemma}
% 
% \begin{proof}If $\alpha$ is any node and $1\le k\le n$ then by the binomial
%   theorem the polynomial
%   $$(x+q)^{\cont_\t(k)}-(x+q)^{\cont(\alpha)}
%   =(x+q)^{\cont(\alpha)}\((x+q)^{\cont_\t(k)-\cont(\alpha)}-1\)$$
%   is a unit in $\O$ if and only if the constant term is non-zero or,
%   equivalently, if and only if $\res(\alpha)\ne\res_\t(k)$, or
%   $\alpha\notin\LAdd_\t(k)$. The result follows by 
%   Definition~\ref{tableau degree}.
% \end{proof}

 The next result connects the graded representation theory of $\H$ with the
 seminormal basis.

\begin{Prop}\label{modular reduction}
  Suppose that $e(\bi)\ne0$, for some $\bi\in I^n$ and let
  $$e(\bi)^\O:=\sum_{\s\in\Std(\bi)}\frac1{\gamma_\s}f_{\s\s}\in\HK.$$
  Then $e(\bi)^\O\in\HO$ and $e(\bi)=e(\bi)^\O\otimes_\O1_K$.
\end{Prop}

\begin{proof} It is shown in \cite[Lemma~4.2]{M:seminormal} that
  $e(\bi)^\O$ is an element of $\HO$. Therefore, we can reduce $e(\bi)^\O$
  modulo the maximal ideal~$\m$ of~$\O$ to obtain an element of $\H$: let
  $\hat e(\bi)=e(\bi)^\O\otimes_\O1_K$. Then $\set{\hat e(\bj)|\bj\in I^n}$ is a
  family of pairwise orthogonal idempotents in $\H$ such that
  $1_{\H}=\sum_\bj \hat e(\bi)$ by \cite[Cor.~4.7]{M:seminormal}.

  As in \cite[Defn.~4.3]{M:seminormal}, for every pair $(\s,\t)$ of standard
  tableaux of the same shape define
  $g_{\s\t}=\hat e(\bi^\s)m_{\s\t}\hat e(\bi^\t)$. Then $\{g_{\s\t}\}$ is a
  (cellular) basis of $\H$ by \cite[Theorem~4.5]{M:seminormal}. Moreover,
  by \cite[Prop.~4.4]{M:seminormal}, if $1\le k\le n$ then in $\H$
  $$g_{\s\t}(L_k-\res_\t(k))
        =\sum_{\substack{(\u,\v)\gdom(\s,\t)\\\u\in\Std(\bi^\s)
               \text{ and }\v\in\Std(\bi^\t)}} r_{\u\v}g_{\u\v},$$
  for some $r_{\u\v}\in K$. It follows that $g_{\s\t}(L_k-\res_\t(k))^N=0$ 
  for $N\gg 0$. Therefore,
  $$\HH_\bi=\sum_{\substack{\u\text{ standard}\\\v\in\Std(\bi)}}Kg_{\u\v}
           =\H \hat e(\bi).$$
  Hence, $e(\bi)=\hat e(\bi)$ by Lemma~\ref{idempotents}(a) as required.
\end{proof}

\subsection{Positive tableaux}
The KLR idempotents $e(\bi)$ in the presentation of $\R\cong\H$ hide a
lot of important information about these algebras.
Proposition~\ref{modular reduction} gives us a way of accessing this
information.

If $\bi=(i_1,\dots,i_n)\in I^n$ then set $\bi_k=(i_1,\dots,i_k)$ so
that $\bi_k\in I^k$, for $1\le k\le n$.

\begin{Defn} \label{positive}
  Suppose that $\s\in\Std(\bi)$, for $\bi\in I^n$. Then
  $\s$ is \textbf{positive} if
  \begin{enumerate}
      \item $\LRem_\s(k)=\emptyset$, for $1\le k\le n$,  and
      \item if $\LAdd_\s(k)\ne\emptyset$, for some $k$, then $\alpha\in\LAdd_\s(k)$ whenever 
	  $\alpha$ is an $i_k$-node which is below $\s^{-1}(k)$ such that $\alpha$ is an addable 
	  node for some tableau $\t\in\Std(\bi_{k-1})$ with $\t\gedom\s_{k-1}$.
  \end{enumerate}
  If $\s$ is a positive tableau define
  $y_\s=\Prod_{k=1}^n y_k^{|\LAdd_\s(k)|}\in\H.$
\end{Defn}

Using the relations in $\R$ it is not clear that $y_\s$ is non-zero
whenever $\s$ is positive.  We show that this is always the case in
Theorem~\ref{positive expansion} below.

By definition, $\deg\s\ge0$ whenever $\s$ is positive. The converse is
false because there are many standard tableau $\t$ which are not
positive such that $\deg\t\ge0$.

\begin{Example}[s]\label{positive example}
    (a) Suppose that $e=3$, $\ell=1$ and $\bi=(0,1,2,2,0,1,1,2,0)$. Then 
    the positive tableaux in $\Std(\bi)$ are:
    $$\tab(123,456,789),\quad \tab(123568,49,7),\quad \tab(1235689,4,7).$$\\
    (b) Suppose that $e=3$, $\ell=1$ and let $\t=\tab(124567,3)$. Then
     $\deg\t=0$, however, the tableau~$\t$ is not positive.\\
    (c) Suppose that $e=2$, $\ell=2$, $\charge=(0,1)$ and that
    $$\s=\Big(\tab(14), \tab(2,3)\Big)\quad\text{and}\quad
      \t=\Big(\tab(1), \tab(234)\Big).$$
      Then $\s$ is not a positive tableau because $\t_3\gdom\s_3$ but 
      $\alpha=(1,2,3)=\t^{-1}(4)$ is not an addable node of $\s$.\\
    (d) Suppose that $e=2$, $\ell=2$, $\charge=(8,0)$ and that
    $$\t=\Bigg(\ \tab(1245,3,7),\tab(6)\Bigg)\quad\text{and}\quad
      \s=\Bigg(\ \tab(12,37,4,5),\tab(6)\Bigg).$$ 
    Then $\s$ and $\t$ both belong to $\Std(\bi)$ and
    $\LRem_\s(k)=\emptyset$, for $1\le k\le 7$. However, $\s$ is not a
    positive tableau because the node $(3,1,1)=\t^{-1}(7)$ is below
    $(2,2,1)=\s^{-1}(7)$ and $(3,1,1)$ is not an addable node of $\s_6$.
\end{Example}

Recall from section~3.2 that if $\blam\in\Multiparts$ then $\tlam$ is the
unique standard $\blam$-tableau such that $\tlam\gedom\t$, for all
$\t\in\Std(\blam$). The tableaux $\tlam$ are the most important
examples of positive tableaux.

\begin{Lemma}\label{positive tlam}
    Suppose that $\blam\in\Multiparts$. Then $\tlam$ is positive.
\end{Lemma}

\begin{proof} By definition, $\LRem_{\tlam}(k)=\emptyset$ for $1\le k\le
  n$, so it remains to check condition~(b) in
  Definition~\ref{positive}. Let $\beta=(r,c,l)$ be the lowest removable node of
  $\blam$, so that $\tlam(\beta)=n$. By induction on $n$ it suffices to show 
  that $\alpha\in\LAdd_{\tlam}(n-1)$ whenever $\alpha=(r',c',l')$ is below
  $\beta$ and there exists a standard tableau $\t\in\Std(\ilam_{n-1})$ such
  that $\t\gdom\tlam_{n-1}$ and $\alpha\in\LAdd_\t(n-1)$.

  Let $\bmu=\Shape(\t)$. Since $\t\gdom\tlam_{n-1}$ we have that
  $\mu^{(k)}=(0)$ for $k>l$. Consequently, $\alpha\in\LAdd_{\tlam}(n-1)$ if
  $l'>l$. As $\alpha$ is below $\beta$ this leaves only the case when
  $l'=l$ in which case we have that $r'>r$. Since $\t\gdom\tlam_{n-1}$ this
  forces $\alpha=(r+1,1,l)$ to be the addable node of $\blam$ in first
  column of the row directly below $\beta$, so
  $\alpha\in\LAdd_{\tlam}(n-1)$ as required.
\end{proof}

Suppose that $\s$ is a positive tableau. To work with $e(\bi^{\s})y_\s$ we have
to choose the correct lift of it to $\HO$. Perhaps surprisingly, we
choose a lift which depends on the tableau $\s$ rather than choosing a
single lift for each of the homogeneous elements $y_1,\dots,y_n$.

\begin{Defn}\label{ys def}
    Suppose that $\bi\in I^n$ and $\s\in\Std(\bi)$ is a positive
    tableau. Define $y_\s^\O=y_{\s,1}^\O\dots y_{\s,n}^\O$, an element of $\HO$, where
    $$ y_{\s,k}^\O=\begin{cases} 
        \Prod_{\alpha\in\LAdd_{\s}(k)}\Big(1-\frac1{\cont(\alpha)}L_k\Big),&\text{if }q\ne1,
	\\[5mm]
        \Prod_{\alpha\in\LAdd_{\s}(k)}\Big(L_k-\cont(\alpha)\Big),&\text{if }q=1,
    \end{cases}
    $$
    for $k=0,\dots,n$ (by convention, empty products are~$1$).
\end{Defn}

By definition, $y_\s^\O\in\HO$. Moreover, 
$e(\bi^{\s})y_\s=e(\bi^{\s})^{\O}y_\s^\O\otimes_\O 1_K\in\H$.

The following Lemma in the case $\s=\tlam$ is the key to the main results in this paper.

\begin{Lemma}\label{ftt ys}
    Suppose that $\bi\in I^n$ and that $\s,\t\in\Std^+(\bi)$ and that $\s$ is positive. Then:
    \begin{enumerate}
      \item If $\t=\s$ then
        $f_{\s\s}y_\s^\O = u^\O_\s \gamma_{\s} f_{\s\s}$,
        for some unit $u_\s^\O\in\O$.
      \item If $\t\ne\s$ then there exists an element $u_\t\in\O$ such
        that
       $$f_{\t\t}y_\s^\O=\begin{cases}
           u_\t f_{\t\t},& \text{if $\t\gdom\s$},\\
           0,&\text{otherwise.}
     \end{cases}$$
\end{enumerate}
\end{Lemma}

\begin{proof}
    By (\ref{Murphy action}), if $1\le k\le n$ then
    $f_{\t\t}L_k=\cont_\t(k)f_{\t\t}$ in $\HK$, so $f_{\t\t}y_\s^\O$ is a
    scalar multiple of $f_{\t\t}$ and it remains to determine this
    multiple. 

    (a) Observe that $\Rem_{\s}^{\Lambda}(k)=\emptyset$, for $1\le k\le n$,
    because $\s$ is a positive tableau.  Further, if
    $\alpha\in\Add_{\s}(k)$ and
    $\alpha\notin\LAdd_\s(k)$ then the factor that $\alpha$
    contributes to $\gamma_{\s}$ is a unit in $\O$.  Therefore, if $q\ne1$
    then applying Definition~\ref{gamma} and Definition~\ref{ys def} shows that
    $$f_{\s\s}y_\s^\O=\prod_{k=1}^n\prod_{\alpha\in\LAdd_{\s}(k)}
    \Big(1-\frac{\cont_\s(k)}{\cont(\alpha)}\Big)\cdot f_{\s\s}
             =u_{\s}^\O \gamma_{\s}f_{\s\s},$$
    for some invertible element $u_{\s}^\O\in\O$, proving (a). If $q=1$
    then the proof is similar.

    (b) Suppose that $1\le k\le n$. Then we claim that
    $$ f_{\t\t}y_{\s,1}^\O\dots y_{\s,k}^\O
         =\begin{cases}
             u_{\t,k}f_{\t\t}, &\text{if }\t_k\gedom\s_k,\\
               0,&\text{otherwise},
          \end{cases}$$
    for some  $u_{\t,k}\in\O$. If $k=0$ then there is nothing to prove so
    we may assume by induction that the claim is true for
    $f_{\t\t}y_{\s,1}^\O\dots y_{\s,k}^\O$ and consider
    $f_{\t\t}y_{\s,1}^\O\dots y_{\s,k+1}^\O$.

    If $\t_k\not\gedom\s_k$ then, by induction, both sides of the claim
    are zero, so we may assume that $\t_k\gedom\s_k$ Let
    $\rho=\t^{-1}(k+1)$ be the node labeled by $k+1$ in $\t$ and $\beta$ be
    the node labeled by $k+1$ in $\s$.

    It remains to show that $f_{\t\t}y_{\s,1}^\O\dots y_{\s,k+1}^\O=0$ when
    $\rest\t{k+1}\not\gedom\rest\s{k+1}$. As $\rest\t k\gedom\rest\s k$ this 
    can happen only if~$\rho$ is below
    $\beta$. However, since $\s$ is positive and $\res(\s)=\res(\t)$, 
    every addable $i_{k+1}$-node of
    $\rest\t k$ below~$\beta$ is an addable node of~$\rest\s k$.
    Hence, $\rho\in\LAdd_\s(k+1)$ and, consequently,
    $\cont_\t(k+1)=\cont(\alpha)$, for some $\alpha\in\LAdd_\s(k+1)$.
    Therefore, the coefficient of $f_{\t\t}$ in $f_{\t\t}
    y_{\s,1}^\O\dots y_{\s,k+1}^\O$ is zero, as we needed to show. This completes the proof of the Lemma.
\end{proof}

Recall the definition of positive tableau from Definition~\ref{positive}.

\begin{Theorem}\label{positive expansion}
    Suppose that $\bi\in I^n$ and that $\s\in\Std(\bi)$ is a positive
    tableau. Then there
    exists a non-zero scalar $c\in K$ such that
    $$e(\bi) y_\s=cm_{\s\s}+\sum_{(\u,\v)\gdom(\s,\s)}r_{\u\v}m_{\u\v},$$
    some $r_{\u\v}\in K$. In particular, $y_\s$
    is a non-zero homogeneous element of $\H$ of degree $2\deg\s$.
\end{Theorem}

\begin{proof}
    To prove the theorem we work in $\HO$ and in $\HK$. By
    Lemma~\ref{ftt ys}, inside $\HK$ we have
    $$e(\bi)^\O y_\s^\O
      =\sum_{\t\in\Std(\bi)}\frac1{\gamma_\t}f_{\t\t}y_\s^\O
      = u_{\s}^\O f_{\s\s}
       + \sum_{\substack{\t\in\Std(\bi)\\\t\gdom\s}}
       \tfrac{u_{\t,n}}{\gamma_\t} f_{\t\t},
    $$
    where $u_{\s}^{\O}$ is invertible in $\O$ and $u_{\t,n}\in\O$, for each
    $\t\gdom\s$.
    Rewriting this equation in terms of the standard basis we see that
    $$e(\bi)^\O y_\s^\O=u_{\s}^\O m_{\s\s}
          +\sum_{(\u,\v)\gdom(\s,\s)} r_{\u\v}m_{\u\v},$$
    for some $r_{\u\v}\in\K$. However,
    $e(\bi)^\O y_\s^\O\in\HO$, by Proposition~\ref{modular reduction}, and 
    $m_{\u\v}\in\HO$ for all $(\u,\v)$.
    So, in fact, $r_{\u\v}\in\O$ for all $(\u,\v)\gdom(\s,\s)$ and
    reducing this equation modulo the maximal ideal~$\m=\pi\O$ gives the first statement in the Theorem.

    Finally, since $y_\s\ne0$ we have that $\deg y_\s=2\deg\s$
    by Definition~\ref{tableau degree} ---~recall that~$\s$ is positive only if
    $\LRem_\s(k)=\emptyset$, for $1\le k\le n$.
\end{proof}

By Lemma~\ref{positive tlam}, the tableau $\tlam$ is positive for any
$\blam\in\Multiparts$. Therefore, we have the following important special
case of Definition~\ref{positive}.

\begin{Defn}\label{ylam defn}
  Suppose that $\blam\in\Multiparts$. Set $\eblam=e(\ilam)$ and 
  $y_\blam=y_{\tlam}$.
  %and $y_{\blam}^{\O}=y_{\tlam}^{\O}$, so that $e_{\blam}y_{\blam}=e_{\blam}^{\O}y_{\blam}^{\O}\otimes_{\O}1_{K}$.
\end{Defn}

As in section~2, if $\blam\in\Multiparts$ let $\Hblam$ be the two-sided
ideal spanned by the $m_{\s\t}$, where $\s,\t\in\Std(\bmu)$ for some
$\bmu\in\Multiparts$ with $\bmu\gdom\blam$.

Then using Theorem~\ref{positive expansion} we obtain:

\begin{cor}\label{ylam}
  Suppose that $\blam\in\Multiparts$. Then $y_\blam$ is a non-zero
  homogeneous element of degree $2\deg\tlam$. Moreover, there exists a non-zero
  scalar $c_\blam\in K$ such that
  $\eblam y_\blam \equiv c_\blam m_\blam\pmod\Hblam$.
\end{cor}

Equivalently, $\eblam y_\blam \equiv c_\blam \eblam m_\blam\eblam
\pmod\Hblam$. From small examples it is plausible that $\eblam m_\blam
\eblam\in\L$, for all $\blam\in\Multiparts$.  This would give a partial
explanation for the last result.

%%%%%%%%%%%%%%%%%%%%%%%%%%%%%%%%%%%%%%%%%%%%%%%%%%%%%%%%%
\section{A graded cellular basis of $\H$}
In this section we build on Theorem~\ref{positive expansion} to prove our
Main Theorem which shows that $\H$ is a graded cellular algebra.  Brundan,
Kleshchev and Wang~\cite{BKW:GradedSpecht} have already constructed a
graded Specht module for $\H$. The main result of this section essentially
`lifts' the Brundan, Kleshchev and Wang's construction of the graded
Specht modules to a graded cellular basis of~$\H$.

\subsection{Lifting the graded Specht modules to $\H$} As Brundan and
Kleshchev note~\cite[\S4.5]{BK:GradedDecomp}, it follows directly from
Definition~\ref{relations} that $\H$ has a unique $K$-linear
anti-automorphism $*$ which fixes each of the graded generators. We
warn the reader that, in general, $*$ is different from the
anti-automorphism of~$\H$ determined by the (ungraded) cellular basis
$\{m_{\s\t}\}$.

Inspired partly by Brundan, Kleshchev and
Wang's~\cite[\S4.2]{BKW:GradedSpecht} construction of the graded Specht
modules in the non-degenerate case we make the following definition.

\begin{Defn}\label{psi defn}
    Suppose that $\blam\in\Multiparts$ and $\s,\t\in\Std(\blam)$ and
    \textit{fix} reduced expressions $d(\s)=s_{i_1}\dots s_{i_k}$ and
    $d(\t)=s_{j_1}\dots s_{j_m}$ for $d(\s)$ and $d(\t)$,
    respectively. Define
    $$\psi_{\s\t}=\psi_{d(\s)}^*\eblam y_\blam\psi_{d(\t)},$$ 
    where $\psi_{d(\s)}=\psi_{i_1}\dots\psi_{i_k}$ and
    $\psi_{d(\t)}=\psi_{j_1}\dots\psi_{j_m}$.
\end{Defn}

An immediate and very useful consequence of this definition and the
homogeneous relations of $\H$ is the following.

\begin{Lemma}\label{psi weights}
  Suppose that $\s,\t\in\Std(\blam)$, for $\blam\in\Multiparts$, and that
  $\bi,\bj\in I^n$.
  Then
  $$e(\bi)\psi_{\s\t}e(\bj)
  =\begin{cases}
    \psi_{\s\t},&\text{if $\res(\s)=\bi$ and }\res(\t)=\bj,\\
    0,&\text{otherwise.}
  \end{cases}$$
\end{Lemma}

The next two results combine Corollary~\ref{ylam} with Brundan, Kleshchev
and Wang's results for the graded Specht modules to describe the
homogeneous elements $\psi_{\s\t}$. 

\begin{Lemma}[\protect{cf.~\cite[Cor.~3.14]{BKW:GradedSpecht}}]
  \label{psi degrees}
  Suppose that $\blam\in\Multiparts$ and $\s,\t\in\Std(\blam)$.
  Then $$\deg\psi_{\s\t}=\deg\s+\deg\t.$$
\end{Lemma}

\begin{proof}By \cite[Cor.~3.14]{BKW:GradedSpecht}, if
  $d(\s)=s_{i_1}\dots s_{i_k}$ is a reduced expression for $d(\s)$ then
  $\deg\s-\deg\tlam=\deg(\eblam \psi_\s)$. Therefore,
  $$\deg\psi_{\s\t}=\deg(\psi_\s^* \eblam y_\blam\psi_\t)
         =\deg(\eblam\psi_\s)+\deg y_\blam+\deg(\eblam\psi_\t)
         =\deg\s+\deg\t,$$
   where the last equality follows because
   $\deg y_\blam=2\deg\tlam$ by Corollary~\ref{ylam}.
\end{proof}

We note that it is possible to prove Lemma~\ref{psi degrees} directly by
induction on the dominance ordering on standard tableaux.  We now show that
$\psi_{\s\t}$ is non-zero.

\begin{Lemma}[\protect{cf.~\cite[Prop.~4.5]{BKW:GradedSpecht}}]
  \label{psi triangular}
  Suppose that $\blam\in\Multiparts$ and that $\s,\t\in\Std(\blam)$. Then
  there exists a non-zero scalar $c\in K$, which does not depend upon the
  choice of reduced expressions for $d(\s)$ and $d(\t)$, such that
  $$\psi_{\s\t}=cm_{\s\t} +\sum_{(\u,\v)\gdom(\s,\t)}r_{\u\v}m_{\u\v},$$
  for some $r_{\u\v}\in K$.
\end{Lemma}

\begin{proof}
    This is a consequence of Corollary~\ref{ylam} and
    \cite[Theorem~4.10a]{BKW:GradedSpecht} when $q\ne1$. We sketch in
    general because this result is central to this paper.

    Let $d(\s)=s_{i_1}\dots s_{i_k}$ and $d(\t)=s_{j_1}\dots s_{j_m}$
    be the reduced expressions for~$d(\s)$ and~$d(\t)$,
    respectively, that we fixed in Definition~\ref{psi defn}.

  By Corollary~\ref{ylam}, $\eblam y_\blam$ is a homogeneous element
  of $\H$ and 
  \begin{align*}  
      \eblam y_\blam \psi_{d(\t)}
      &\equiv c_\blam m_\blam\pmod\Hblam.\\
  \intertext{Using Theorem~\ref{BK main}
  and the homogeneous relations of~$\H$ it is easy to prove that
  $\eblam\psi_{d(\t)}$ is equal to a linear combination of
  terms of the form $\eblam f_w(y)T_{w}$, where 
  $f_w(y)\in K[y_1,\dots,y_n]$ for some $w\in\BS_n$ with $w\leq d(\t)$, 
  and where $f_{d(\t)}(y)$ is invertible. By (\ref{Murphy action}), 
  $m_\blam y_r\equiv m_\blam \eblam y_r\equiv0\pmod\Hblam$, for 
  $1\le r\leq n$. Now if $w\in\Sym_n$ then, modulo $\HH_n^{\rhd\lam}$, $m_{\lam}T_w$ 
  can be written as a linear combination of elements of the form 
  $m_{\tlam\v}$, where $\v\in\Std(\lam)$ and $d(\v)\leq w$, by
  Theorem~\ref{standard basis}. Therefore, 
  just as in~\cite[Prop.~4.5]{BKW:GradedSpecht}, we obtain} 
  \eblam y_\blam \psi_{d(\t)}
     &\equiv c'm_{\tlam\t}+\sum_{\substack{\v\in\Std(\blam)\\\v\gdom\t}}
     r_\v m_{\tlam\v}\\
  \intertext{for some $c',r_\v\in K$ with $c'\ne0$. The scalar $c'$ depends
  only on $\t$ and $\blam$, and not on the choice of reduced expression
  for $d(\t)$, by \cite[Prop.~2.5(i)]{BKW:GradedSpecht}.  Similarly,
  multiplying the last equation on the left with $\psi_{d(\s)}^*\eblam$,
  and again using (\ref{Murphy action}) and the fact that $\{m_{\u\v}\}$ is
  a cellular basis, we obtain}
   \psi_{\s\t}&\equiv c m_{\s\t}
         +\sum_{\substack{\u,\v\in\Std(\blam)\\(\u,\v)\gdom(\s,\t)}}
         r_{\u\v}m_{\u\v}\pmod\Hblam
         \end{align*}
    for some $r_{\u\v}\in K$ and some non-zero scalar $c\in K$ which
    depends only on $d(\s)$, $d(\t)$ and $\blam$. This completes the proof. 
\end{proof}

Recall from section~4.3 that $\Hblam$ is the two-sided ideal of $\H$
with basis the of standard basis elements $\{m_{\u\v}\}$, where
$\u,\v\in\Std(\bmu)$ and $\bmu\gdom\blam$.

\begin{cor}\label{graded ideals}
  Suppose that $\blam\in\Multiparts$. Then $\Hblam$ is a homogeneous
  two-sided ideal of $\H$ with basis
  $\set{\psi_{\u\v}|\u,\v\in\Std(\bmu), 
           \text{ for }\bmu\in\Multiparts \text{ with }\bmu\gdom\blam}$.
\end{cor}

As the next example shows,  in general, the elements $\psi_{\s\t}$ 
depend upon the choice of the reduced expressions for $d(\s)$ and $d(\t)$.

\begin{Example}\label{dependence}
  Suppose that $e=3$, $\Lambda=\Lambda_0$ and $n=9$ so that we are
  considering the Iwahori-Hecke algebra of $\Sym_9$ at a third root of
  unity (for any suitable field). Take $\lambda=(4,3,1^2)$ and set
  $$\t=\tab(1239,468,5,7)\qquad\text{and}\qquad
    \u=\tab(1237,468,5,9).$$
  Then $d(\t)=s_4s_5s_7s_6s_5s_7s_8s_7=s_4s_5s_7s_6s_5s_8s_7s_8$.
  Now, $\res_\t(7)=\res_\t(9)$ so applying the last
  relation in Definition~\ref{relations} (the graded analogue of the braid
  relation),
  $$e_\lambda y_\lambda \psi_4\psi_5\psi_7\psi_6\psi_5\psi_7\psi_8\psi_7
        =e_\lambda y_\lambda\(\psi_4\psi_5\psi_7\psi_6\psi_5\psi_8\psi_7\psi_8
        +\psi_4\psi_5\psi_7\psi_6\psi_5\).$$
  Consequently, if $\s\in\Std(\lambda)$ and we define $\psi_{\s\t}$ using
  the first reduced expression for~$d(\t)$ above and $\hat\psi_{\s\t}$ using
  the second reduced expression then
  $\psi_{\s\t}=\hat\psi_{\s\t}+\psi_{\s\u}$.
  Therefore, different choices of reduced expression for $d(\t)$ can give
  different elements~$\psi_{\s\t}$, for any $\s\in\Std(\lambda)$.
\end{Example}

We do not actually need the next result, but given
Example~\ref{dependence} it is reassuring. Brundan, Kleshchev and Wang
prove an analogue of this result as part of their
construction of the graded Specht modules
\cite[Theorem~4.10]{BKW:GradedSpecht}. They have to work much harder,
however, as they have to simultaneously prove that the grading on
their modules is well-defined.

\begin{Lemma}[\protect{cf.~\cite[Theorem.~4.10a]{BKW:GradedSpecht}}]
  \label{psi error terms}
  Suppose that $\psi_{\s\t}$ and $\hat\psi_{\s\t}$ are defined using
  different reduced expressions for $d(\s)$ and $d(\t)$, where
  $\s,\t\in\Std(\blam)$ for some $\blam\in\P$. Then
  $$\psi_{\s\t}-\hat\psi_{\s\t}
      =\sum_{(\u,\v)\gdom(\s,\t)} s_{\u\v}\psi_{\u\v},$$
  where $s_{\u\v}\ne0$ only if $\res(\u)=\res(\s)$,
  $\res(\v)=\res(\t)$ and $\deg\u+\deg\v=\deg\s+\deg\t$.
\end{Lemma}

\begin{proof}Using two applications of (\ref{psi triangular}), we can write
  $$
  \psi_{\s\t}-\hat\psi_{\s\t}=\sum_{(\u,\v)\gdom(\s,\t)} r_{\u\v}m_{\u\v}
                             =\sum_{(\u,\v)\gdom(\s,\t)} s_{\u\v}\psi_{\u\v},
  $$
  for some $r_{\u\v},s_{\u\v}\in K$. Multiplying on the left and right by
  $e(\bi^\s)$ and $e(\bi^\t)$, respectively, and using
  Lemma~\ref{psi weights}, shows that $s_{\u\v}\ne0$ only if
  $\res(\u)=\res(\s)$ and $\res(\v)=\res(\t)$. Finally, by
  Lemma~\ref{psi triangular}, the $\psi_{\u\v}$ appearing on the right
  hand are all linearly independent and $\psi_{\s\t}$ and
  $\hat\psi_{\s\t}$ are non-zero homogeneous elements of the same
  degree by Lemma~\ref{psi degrees}.  Therefore, so if $s_{\u\v}\ne0$
  then $\deg\u+\deg\v=\deg\psi_{\u\v}=\deg\psi_{\s\t}=\deg\s+\deg\t$,
  as required.
\end{proof}

We can now prove the main result of this paper. The existence of a graded
cellular basis for $\H$ was conjectured by Brundan, Kleshchev and
Wang~\cite[Remark~4.12]{BKW:GradedSpecht}.  See Definition~\ref{psi defn}
for the definition of the elements $\psi_{\s\t}$, for $\s,\t\in\Std(\blam)$.

\begin{Theorem}[Graded cellular basis]\label{psi basis}
  The algebra $\H$ is a graded cellular algebra with weight poset
  $(\Multiparts,\gedom)$ and graded cellular basis
  $\set{\psi_{\s\t}|\s,\t\in\Std(\blam) \text{ for }\blam\in\Multiparts}$.
  In particular, $\deg\psi_{\s\t}=\deg\s+\deg\t$, for all
  $\s,\t\in\Std(\blam)$, $\blam\in\Multiparts$.
\end{Theorem}

\begin{proof}By (\ref{psi triangular}), the transition matrix between
    the set $\{\psi_{\s\t}\}$ and the standard basis $\{m_{\s\t}\}$ is
    an invertible triangular matrix (when suitably ordered!).
    Therefore, $\{\psi_{\s\t}\}$ is a basis of $\H$ giving (GC$_1$) from
    Definition~\ref{graded cellular def}. By definition $\psi_{\s\t}$
    is homogeneous and $\deg\psi_{\s\t}=\deg\s+\deg\t$, by
    Lemma~\ref{psi degrees}, establishing (GC$_d$). 
  
  To prove (GC$_3$), recall that $*$ is the unique anti-isomorphism of
  $\H$ which fixes each of the graded generators. By definition,
  $(\eblam y_\blam)^*=\eblam y_\blam$ since $\eblam$ and $y_\blam$
  commute. Therefore, $\psi_{\s\t}^*=\psi_{\t\s}$, for all $\s$ and
  $\t$. Consequently, the anti-automorphism of~$\H$ induced by the basis
  $\{\psi_{\s\t}\}$, as in (GC$_3$), coincides with the
  anti-isomorphism~$*$. In particular, (GC$_3$) holds.

  It remains then to check that the basis $\{\psi_{\s\t}\}$ satisfies
  (GC$_2$), for $\s,\t\in\Std(\blam)$ and $\blam\in\Multiparts$. By
  definition, $\psi_{\s\t}=\psi^*_{d(\s)}\psi_{\tlam\t}$.
  Suppose that $h\in\H$.  Using Lemma~\ref{psi triangular} twice, together
  with Corollary~\ref{graded ideals} and the fact that $\{m_{\u\v}\}$ is 
  a cellular basis of~$\H$, we find
  \begin{align*}
      \psi_{\s\t}h&=\psi_{d(\s)}^*\psi_{\tlam\t}h
      \equiv\psi_{d(\s)}^*\sum_{\v\gedom\t}r_\v m_{\tlam\v}h\pmod{\Hblam}\\
      &\equiv\psi_{d(\s)}^*\sum_{\v\in\Std(\blam)}s_\v m_{\tlam\v}\pmod{\Hblam}\\
      &\equiv\psi_{d(\s)}^*\sum_{\v\in\Std(\blam)}t_\v \psi_{\tlam\v}\pmod{\Hblam}\\
      &\equiv\sum_{\v\in\Std(\blam)}t_\v \psi_{\s\v}\pmod{\Hblam}\\
  \end{align*}
  for some scalars $r_\v, s_\v, t_\v\in K$. Hence,
  $\{\psi_{\s\t}\}$ is a graded cellular basis and~$\H$ is a
  graded cellular algebra, as required.
\end{proof}

Applying Corollary~\ref{graded dimension}, we obtain the graded
dimension of~$\H$
$$\Dim\H=\sum_{\blam\in\Multiparts}\sum_{\s,\t\in\Std(\blam)}
           t^{\deg\s+\deg\t}.$$
This result is due to Brundan and
Kleshchev~\cite[Theorem~4.20]{BK:GradedDecomp}. See also
\cite[Remark~4.12]{BKW:GradedSpecht}. This can be further refined to compute
$\Dim e(\bi)\H e(\bj)$, for $\bi,\bj\in I^n$, using Lemma~\ref{psi weights}.

\subsection{The graded Specht modules}
Now that $\{\psi_{\s\t}\}$ is known to be a graded cellular basis we
can define the graded cell modules $S^\blam$ of~$\H$, for~$\blam\in\Multiparts$. 
\begin{Defn}[Graded Specht modules] Suppose that $\blam\in\Multiparts$. The
  \textbf{graded Specht module} $S^\blam$ is the graded cell module
  associated with $\blam$ as in Definition~\ref{graded cells}.
\end{Defn}

Thus, $S^\blam$ has basis $\set{\psi_\t|\t\in\Std(\blam)}$ and the
action of $\H$ on $S^\blam$ comes from its action on
$\HH^{\gedom\blam}_n/\Hblam$.

In the absence of a graded cellular basis, Brundan, Kleshchev and
Wang~\cite{BKW:GradedSpecht} have already defined a graded Specht module
$S^\blam_{BKW}$, for $\blam\in\Multiparts$ (when $q\ne1)$. The two notions
of graded Specht modules coincide.

\begin{cor}
  Suppose that $\blam\in\Multiparts$. Then $S^\blam\cong S^\blam_{BKW}$ as
  $\Z$-graded $\H$-modules.
\end{cor}

\begin{proof}
  Brundan, Kleshchev and Wang~\cite{BKW:GradedSpecht} actually define the
  graded left module $S^{*\blam}_{BKW}$, however, it is an easy exercise to
  switch their notation to the right. Mirroring the notation of
  \cite[\S4.2]{BKW:GradedSpecht}, set
  $\dot v_\blam= \eblam y_\blam+\Hblam=\psi_{\tlam\tlam}+\Hblam$. By
  Theorem~\ref{psi basis} the graded right module $\dot v_\blam\H$ has basis
  $\set{\dot v_\blam\psi_{d(\t)}|\t\in\Std(\blam)}$. Comparing this construction
  with~\cite[\S4.2]{BKW:GradedSpecht} and Definition~\ref{graded cells}
  it is immediate that
  $$S^\blam_{BKW}\cong \dot v_\blam\H\<-\deg\tlam\>\cong S^\blam.$$
  In the notation of \cite{BKW:GradedSpecht}, the first isomorphism is
  given by $v_\t\mapsto \dot v_\blam\psi_{d(\t)}$, for $\t\in\Std(\blam)$.
  There is a degree shift for the middle term because
  $\deg \dot v_\blam=2\deg\tlam$ by Corollary~\ref{ylam}.
\end{proof}

By Lemma~\ref{psi triangular} and Corollary~\ref{graded ideals}, the
ungraded module $\underline{S}^\blam$ coincides with the ungraded Specht
module determined by the standard basis (Theorem~\ref{standard basis}),
because the transition matrix between the graded cellular basis and the
standard basis is unitriangular.

Let $\dot D^\bmu$ be the ungraded simple $\H$-module which is defined using
the standard basis of~$\H$, for $\bmu\in\Multiparts$. Define a
multipartition $\bmu$ to be \textbf{$\Lambda$-Kleshchev} if $\dot D^\bmu\ne0$.
Although we will not need it, there is an explicit combinatorial
characterization of the $\Lambda$-Kleshchev multipartitions; see
\cite{Ariki:class} or \cite[(3.27)]{BK:GradedDecomp} (where they are called
\textit{restricted multipartitions}).

By Theorem~\ref{graded simples}, and the remarks of the last paragraph, the
graded irreducible $\H$-modules are labeled by the
$\Lambda$-Kleshchev multipartitions of~$n$. Notice, however, that this does
not immediately imply that $D^\bmu$ is non-zero if and only if~$\bmu$ is a
$\Lambda$-Kleshchev multipartition: the problem is that the homogeneous
bilinear form on the graded Specht module, which is induced by the graded
basis (see Lemma~\ref{symmetric form}), could be different to the bilinear
form on the ungraded Specht module, which is induced by the standard basis.
Our next result shows, however, that these two forms are essentially
equivalent because their radicals coincide.

The following result is almost the same as
\cite[Theorem~5.10]{BK:GradedDecomp}.

\begin{cor}
  Suppose that $\bmu\in\Multiparts$. Then $\dot D^\bmu=\underline{D}^\bmu$,
  for all $\bmu\in\Multiparts$. Consequently, $D^\bmu\ne0$ if and only if
  $\bmu$ is a~$\Lambda$-Kleshchev multipartition.
\end{cor}

\begin{proof}
  We argue by induction on dominance. If $\bmu$ is minimal in the dominance
  order then $D^\mu=S^\bmu$ and $\dot D^\bmu=\underline{S}^\bmu$ by
  Lemma~\ref{graded decomp}(c). Hence, $\dot D^\bmu=\underline{D}^\bmu$ in this case. 
  Now suppose that $\bmu$ is not minimal with respect to dominance.
  Using Lemma~\ref{graded decomp}(c) again, $D^\bmu=0$
  if and only if every composition factor of $S^\bmu$ is isomorphic to
  $D^\bnu$ for some multipartition $\bnu$ with $\bmu\gdom\bnu$.
  Similarly, $\dot D^\bmu=0$ if and only if every composition factor of
  $\underline{S}^\bmu$ is isomorphic to $\dot D^\bnu$, where
  $\bmu\gdom\bnu$. By induction, $\dot D^\bnu=\underline{D}^\bnu$ so the
  result follows.
\end{proof}

\subsection{The blocks of $\H$}
We now show how Theorem~\ref{psi basis} restricts to give a basis for
the blocks, or the indecomposable two-sided ideals, of $\H$. Recall
that $Q_+=\bigoplus_{i\in I}\N\alpha_i$ is the positive root
lattice. Fix $\beta\in Q_+$ with $\sum_{i\in I}(\Lambda_i,\beta)=n$ and
let
$$I^\beta=\set{\bi\in I^n|\alpha_{i_1}+\dots+\alpha_{i_n}=\beta}.$$
Then $I^\beta$ is an $\Sym_n$-orbit of $I^n$ and it is not hard to check
that every $\Sym_n$-orbit can be written uniquely in this way for some 
$\beta\in Q_+$. Define
$$\Hbeta=e_\beta\H,\qquad\text{where }e_\beta=\sum_{\bi\in I^\beta}e(\bi).$$
Then by \cite[Theorem~2.11]{LM:AKblocks} and \cite[Theorem~1]{Brundan:degenCentre}, 
$\Hbeta$ is a block of $\H$.  That is,
$$\H=\bigoplus_{\beta\in Q_+,\,I^\beta\ne\emptyset}\Hbeta.$$
is the decomposition of $\H$ into a direct sum of indecomposable two-sided
ideals. Let
$\Multiparts[\beta]=\set{\blam\in\Multiparts|\ilam\in I^\beta}$. It follows
from the combinatorial classification of the blocks of $\H$ that
$\coprod_{\bi\in
I^\beta}\Std(\bi)=\coprod_{\blam\in\Multiparts[\beta]}\Std(\blam)$. Hence,
by Lemma~\ref{psi weights} and Theorem~\ref{psi basis} we obtain the
following.

\begin{cor}
  Suppose that $\beta\in Q_+$. Then
  $$\set{\psi_{\s\t}|\s,\t\in\Std(\blam)
                           \text{ for }\blam\in\Multiparts[\beta]}$$
  is a graded cellular basis of $\Hbeta$. In particular, $\Hbeta$ is a
  graded cellular algebra.
\end{cor}

\subsection{Integral Khovanov-Lauda--Rouquier algebras} 
The Khovanov-Lauda--Rouquier algebras $\R$ are defined over an arbitrary
commutative integral domain~$R$. So far we have produced a cellular basis
for $\R$ only when $R=K$ is a field of characteristic $p\ge0$ such that
either $e=0$ or $e>0$ and $\gcd(e,p)=1$ or $e=p$. By 
Theorem~\ref{BK main} this corresponds to the cases where $\R$ is
isomorphic to a degenerate or non-degenerate Hecke algebra.  In this
section we extend Theorem~\ref{psi basis} to a more general class of rings. 

Throughout this section, let $\R(\Z)$ be the Khovanov-Lauda--Rouquier
algebra of type $\Gamma=\Gamma_e$ defined over $\Z$, where
$e\in\{0,2,3,4,\dots\}$. Let $\bR(\Z)$ be the
torsion free part of $\R(\Z)$. If $\O$ is any
commutative integral domain let $\R(\O)$ be the
Khovanov-Lauda--Rouquier algebra over~$\O$.

The following result is implicit in \cite[Theorem~6.1]{BK:GradedKL}. It
arose out of discussions with Alexander Kleshchev.

\begin{Lemma}\label{integral}
  \begin{enumerate}
    \item Suppose that $e=0$ or that $e$ is prime. Then $\R(\Z)=\bR(\Z)$ is a 
      free $\Z$-module of rank~$\ell^nn!$.
      \item Suppose that $e>0$ is not prime. Then $\R(\Z)$ has $p$-torsion,
        for a prime~$p$, only if $p$ divides $e$.
  \end{enumerate}
\end{Lemma}

\begin{proof}
  First, observe that by Theorem~\ref{BK main} 
  $$\rank\bR(\Z)=\dim_{\Q}(\R(\Z)\otimes_{\Z}\Q)=\dim_{\Q}\R(\Q)=\ell^nn!,$$
  where we take $q$ to be a primitive $e^{\text{th}}$ root of unity in 
  $\mathbb C$ if $e\ne0$ and not a root of unity if $e=0$. 

  Next suppose that $e=0$ and $p$ is any prime. Let $K$ be an infinite
  field of characteristic~$p$ and let $q\in K$ be a transcendental element
  of $K$. Then $\H\cong\R(K)\cong\R(\Z)\otimes_\Z K$ by Theorem~\ref{BK
  main}, so that $\R(\Z)$ has no $p$-torsion.

  Now suppose that $e>0$ and that $p$ is prime not dividing $e$. Let $K$ be
  a field of characteristic~$p$ which contains a primitive $e^{\text{th}}$
  root of unity~$q$ and let $\H$ be the non-degenerate cyclotomic Hecke
  algebra with parameters~$q$ and $\bQ_{\Lambda}$. Then
  $\H\cong\R(K)\cong\R(\Z)\otimes_\Z K$ by Brundan and Kleshchev's
  isomorphism Theorem~\ref{BK main}. Hence,
  $\R(\Z)$ has no $p$-torsion. 
  
  Finally, consider the case when $e=p$ is prime and let $K$ be a field of
  characteristic~$p$. Let $\H$ be the degenerate cyclotomic Hecke
  algebra over $K$ with parameters $\bQ_{\Lambda}$. Then 
  $\H\cong\R(K)\cong\R(\Z)\otimes_\Z K$, so
  once again $\R(\Z)$ has no $p$-torsion. Hence, $\R(\Z)$ can have
  $p$-torsion only if $e>0$ is not prime and $p$~divides~$e$.
\end{proof}

The graded cellular basis $\{\psi_{\s\t}\}$ is defined in terms of the
generators of $\R(\Z)$. Moreover, if $e=0$ and $K$ is any field, or if
$e>0$ and $K$ is a field containing a primitive $e^{\text{th}}$ root
of~$1$, then $\{\psi_{\s\t}\otimes 1_K\}$ is a graded cellular basis of the
algebra $\R(K)\cong\H$. Further, if $e=p$ is prime then
$\{\psi_{\s\t}\otimes_\Z1_K\}$ is a graded cellular basis of
$\R(K)\cong\H$ whenever $K$ is a field of characteristic $p$. Hence,
applying Lemma~\ref{integral} and Theorem~\ref{psi basis}, we obtain
our Main Theorem from the introduction.

\begin{Theorem}
  Let $\O$ be a commutative integral domain and suppose that either $e=0$, 
  $e$ is non-zero prime, or that $e\cdot1_\O$ is invertible in $\O$. Then
  $\R(\O)\cong\R(\Z)\otimes_\Z\O$ is a graded
  cellular algebra with graded cellular basis
  $$\set{\psi_{\s\t}\otimes1_\O|\s,\t\in\Std(\blam)
                   \text{ and }\blam\in\Multiparts}.$$
\end{Theorem}

It seems likely to us that the $\psi$-basis is a graded cellular basis
of~$\R(\Z)$.

%%%%%%%%%%%%%%%%%%%%%%%%%%%%%%%%%%%%%%%%%%%%%%%%%%%%%%%%
\section{A dual graded cellular basis and a homogeneous trace form}
In this section we construct a second graded cellular basis
$\{\psi_{\s\t}'\}$ for the algebras~$\H$ and $\Hbeta$. Using both the
$\psi$-basis and the $\psi'$-basis we then show that $\Hbeta$ is a
graded symmetric algebra, proving another conjecture of Brundan and
Kleshchev~\cite[Remark~4.7]{BK:GradedDecomp}.

\subsection{The dual Murphy basis} The main idea is that the $\psi$-basis
is, via the standard basis $\{m_{\s\t}\}$, built from the trivial
representation of $\H$. The new basis that we will construct is, via the
$\{n_{\s\t}\}$ basis defined below, modeled on the sign representation
of~$\H$.  

\begin{Defn}[\protect{Du and Rui~\cite[(2.7)]{DuRui:branching}}]
  Suppose that $\blam\in\Multiparts$ and $\s,\t\in\Std(\blam)$.
  Define 
  $n_{\s\t}=(-q)^{-\ell(d(\s))-\ell(d(\t))}T_{d(\s)^{-1}}n_\blam T_{d(\t)}$,
  where
    $$n_\blam=\prod_{s=1}^{\ell-1}
           \prod_{k=1}^{|\lambda^{(1)}|+\dots+|\lambda^{(\ell-s)}|}
        (L_k-q^{\kappa_s})\cdot\sum_{w\in\Sym_\blam}(-q)^{-\ell(w)}T_w.$$
   (The normalization of $n_{\s\t}$ by a power of $-q^{-1}$ is for
   compatibility with the results from \cite{M:gendeg} that we use below.
   The asymmetry in the definitions of the basis elements $m_{\s\t}$ and
   $n_{\s\t}$ arises because the relations $(T_r-q)(T_r+1)=0$, for 
   $1\le r<n$ are asymmetric. Renormalizing these relations to
   $(\hat T_r-v)(\hat T_r+v^{-1})=0$, where $q=v^2$, makes the definition
   of these elements symmetric; see, for example, \cite[\Sect3]{M:tilting}.)
\end{Defn}

It follows from Theorem~\ref{standard basis} that $\{n_{\s\t}\}$ is a
cellular basis of $\H$; see \cite[(3.1)]{M:gendeg}.  We now recall how
$L_1,\dots,L_n$ acts on this basis. To describe this requires some more
notation.

If $\lambda=(\lambda_1,\lambda_2,\dots)$ is a partition then its
\textbf{conjugate} is the partition
$\lambda'=(\lambda'_1,\lambda'_2,\dots)$, where
$\lambda'_i=\#\set{j\ge1|\lambda_j\ge i}$. If $\t$ is a standard
$\lambda$-tableau let $\t'$ be the standard $\lambda'$-tableau given by
$\t'(r,c)=\t(c,r)$. Pictorially, $\lambda'$ and $\t'$ are obtained by
interchanging the rows and the columns of $\lambda$ and $\t$,
respectively.

Similarly, if $\blam=(\lambda^{(1)},\dots,\lambda^{(\ell)})$ is a
multipartition then the \textbf{conjugate multipartition} is the
multipartition $\blam'=(\lambda^{(\ell)'},\dots,\lambda^{(1)'})$. If $\t$
is a standard $\blam$-tableau then the \textbf{conjugate tableau} $\t'$ is
the standard $\blam'$-tableau given by $\t'(r,c,l)=\t(c,r,\ell-l+1)$.

By the argument of~\cite[Prop.~3.3]{M:gendeg}, if
$\s,\t\in\Std(\blam)$ and $1\le k\le n$ then
there exist scalars $r_{\u\v}\in K$ such that
\begin{equation}\label{n-Murphy action}
  n_{\s\t}L_k=\res_{\t'}(k) n_{\s\t}+
                \sum_{(\u,\v)\gdom(\s,\t)}r_{\u\v}n_{\u\v}.
\end{equation}

As in section~4.2, fix a modular system $(\K,\O,K)$ for $\H$. Until noted
otherwise we will work in $\HK$.  Following
Definition~\ref{seminormal basis}, define
$f_{\s\t}'=F_{\s'}n_{\s\t}F_{\t'}$,  for $\s,\t\in\Std(\blam)$,
$\blam\in\Multiparts$. Moreover, by (\ref{n-Murphy action}), if
$\s,\t\in\Std(\blam)$, for $\blam\in\Multiparts$, then
$$f'_{\s\t}=n_{\s\t}+\sum_{(\u,\v)\gdom(\s,\t)}r_{\u\v}n_{\u\v},$$
for some $r_{\u\v}\in K$. Therefore, $\{f_{\s\t}'\}$
is a basis of $\HK$, as was noted in \cite[\Sect3]{M:gendeg}.  

We now retrace our steps from section~4.2 replacing the $f_{\s\t}$ basis with
the $f'_{\s\t}$ basis.

Recall from section~4.2 that if $\alpha=(r,c,l)$ and $\beta=(s,d,m)$ are
two nodes then $\alpha$ is below $\beta$ if either $l>m$, or $l=m$ and
$r>s$.  Dually,  we say that $\beta$ is \textbf{above} $\alpha$. With
this notation we can define a `dual' version of the scalars
$\gamma_\t\in\K$.

\begin{Defn}[cf. Definition~\ref{gamma}]\label{gamma'}
  Suppose that $\blam\in\Multiparts$ and $\t\in\Std(\blam)$.  For
  $k=1,\dots,n$ let $\Add_\t(k)'$ be the set of addable nodes of the
  multipartition $\Shape(\rest\t k)$ which are \textit{above} $\t^{-1}(k)$.
  Similarly, let $\Rem_\t(k)'$ be the set of removable nodes of
  $\Shape(\rest\t k)$ which are \textit{above} $\t^{-1}(k)$. Now define
  $$\gamma'_\t=v^{-\ell(d(\t))-\delta(\blam)}\prod_{k=1}^n
            \dfrac{\prod_{\alpha\in\Add_{\t'}(k)'}
                      \(\cont_{\t'}(k)-\cont(\alpha)\)}%
                  {\prod_{\rho\in\Rem_{\t'}(k)'}
                   \(\cont_{\t'}(k)-\cont(\rho)\)}\quad\in\K.
  $$
\end{Defn}

Suppose that $\bi\in I^n$ and that $\Std(\bi)\ne\emptyset$. Define
$\bi'=\res(\s')$, where $\s$ is any element of $\Std(\bi)$. Then
$\bi'\in I^n$ and $\bi'$ is independent of the choice of~$\s$.

Recall that Proposition~\ref{modular reduction} defines the idempotent
$e(\bi)^\O\in\HO$, for $\bi\in I^n$.

\begin{Lemma}\label{mod II}
  Suppose that $\bi\in I^n$ with $e(\bi)\neq 0$. Then, in $\HO$,
  $$e(\bi')^\O=\sum_{\s\in\Std(\bi)}\frac1{\gamma'_{\s}} f'_{\s\s}.$$
\end{Lemma}

\begin{proof}By the argument of~\cite[Remark~3.6]{M:gendeg}, if $\s\in\Std(\bi)$ then 
  $\frac1{\gamma'_{\s}}f'_{\s\s}=\frac1{\gamma_{\s'}}f_{\s'\s'}$ in~$\HK$. So, 
  the result is just a rephrasing of Proposition~\ref{modular reduction}.
  (Note that $\gamma'_\t$, as defined in Definition~\ref{gamma'}, is
  the specialization at the parameters of~$\HK$ of the element $\gamma'_\t$
  defined in \cite[\Sect3]{M:gendeg}; see the remarks before
  \cite[Prop.~3.4]{M:gendeg}.)
\end{proof}

Definition~\ref{positive} defines a homogeneous element $y_\s\in\H$
for each positive tableau $\s\in\Std(\bi)$, $\bi\in I^n$. To
construct the dual basis we lift $e(\bi')y_\s$ to $\HO$.

\begin{Defn}
    Suppose that $\s\in\Std(\bi)$ is a positive tableau. Let
    $$\LAdd_{\s'}(k)'=\set{\alpha\in\Add_{\s'}(k)'|\res(\alpha)=\res_{\s'}(k)}$$
    and define $(y'_\s)^\O=(y'_{\s,1})^\O\dots(y'_{\s,n})^\O$, where
$$(y'_{\s,k})^\O =\begin{cases}\Prod_{\alpha\in\LAdd_{\s'}(k)'}
                         \Big(1-\frac1{\cont(\alpha)}L_k),&\text{if }q\ne1,\\
           \Prod_{\alpha\in\LAdd_{\s'}(k)'}\Big(L_k-\cont(\alpha)\Big),
           &\text{if }q=1,
\end{cases}
$$
    for $k=1,\dots,n$.
\end{Defn}

Observe that if $\s\in\Std(\bi)$ is a positive tableau then
$e(\bi')y_\s=e(\bi')^\O(y'_\s)^\O\otimes_\O1_K$ because
$|\LAdd_\s(k)|=|\LAdd_{\s'}(k)'|$, for $1\le k\le n$.  Note, however,
that $(y'_\s)^\O\ne y_\s^\O$ in general.

The following two results are analogues of Lemma~\ref{ftt ys} and
Theorem~\ref{positive expansion}, respectively. We leave the details to
the reader because they can be proved by repeating the arguments from
section~4, the only real difference being that Lemma~\ref{mod II} is
used instead of Proposition~\ref{modular reduction}.

\begin{Lemma}\label{dual ftt ys}
    Suppose that $\s,\t\in\Std(\bi)$, where $\bi\in I^n$, and that $\s$
    is a positive tableau. Then:
    \begin{enumerate}
      \item If $\t=\s$ then
        $f'_{\t\t}(y'_\s)^\O = u^\O_\s \gamma'_{\s} f'_{\s\s}$,
        for some unit $u_\s^\O\in\O$.
      \item If $\t\ne\s$ then there exists an element $u_\t\in\O$ such
        that
       $$f'_{\t\t}(y'_\s)^\O=\begin{cases}
           u_\t f'_{\t\t},& \text{if $\t\gdom\s$},\\
           0,&\text{otherwise,}
     \end{cases}$$
\end{enumerate}
\end{Lemma}

As a consequence, we can repeat the proof of Theorem~\ref{positive expansion} to
deduce the following.

\begin{Prop}\label{copositive expansion}
    Suppose that $\s\in\Std(\bi)$ is a positive tableau, for $\bi\in I^n$. 
    Then there exists a non-zero $c\in K$ such that 
    $$ e(\bi')y_\s=c n_{\s\s} +\sum_{(\u,\v)\gdom(\s,\s)}r_{\u\v}n_{\u\v},$$
    for some $r_{\u\v}\in K$.
\end{Prop}

\subsection{The dual graded basis}
If $\blam\in\Multiparts$ then $\tlam$ is a positive tableau by
Lemma~\ref{positive tlam}. Recall that $\eblam=e(\ilam)$. Define
$\ebllam=e(\bi')$, where $\bi=\ilam$. Then as a special case of
Proposition~\ref{copositive expansion}, there is a non-zero $c\in K$ such
that
\begin{equation}\label{dualep}\ebllam y_\blam 
        = cn_\blam+\sum_{(\u,\v)\gdom(\tlam,\tlam)}r_{\u\v}n_{\u\v},
\end{equation}
for some $r_{\u\v}\in K$. This is what we need to define the
dual graded basis of $\H$.

\begin{Defn}\label{psi' defn}
    Suppose that $\blam\in\Multiparts$ and $\s,\t\in\Std(\blam)$ and
    recall that we have fixed reduced expressions
    $d(\s)=s_{i_1}\dots s_{i_k}$ and $d(\t)=s_{j_1}\dots s_{j_m}$ for $d(\s)$
    and $d(\t)$, respectively. Define
    $\psi'_{\s\t}=\psi_{i_k}\dots\psi_{i_1}\ebllam y_\blam
                \psi_{j_1}\dots\psi_{j_m}.$
\end{Defn}

By definition, $\psi'_{\s\t}$ is a homogeneous element of $\H$. Just as
with $\psi_{\s\t}$, the element $\psi'_{\s\t}$ will, in general, depend
upon the choice of reduced expressions for $d(\s)$ and $d(\t)$. Arguing
just as in section~5.1 we obtain the following facts. We leave the details to
the reader.

{\samepage
\begin{Prop}\label{psi' properties}
  Suppose that $\s,\t\in\Std(\blam)$, for some $\lam\in\Multiparts$. Then
  \begin{enumerate}
    \item If $\bi,\bj\in I^n$ then
      $$e(\bi')\psi'_{\s\t}e(\bj')=\begin{cases}
        \psi'_{\s\t},&\text{if $\res(\s)=\bi$ and $\res(\t)=\bj$},\\
        0,&\text{otherwise}.
      \end{cases}$$
    \item $\deg\psi'_{\s\t}=\deg\s+\deg\t$.
    \item
      $\psi'_{\s\t}=cn_{\s\t}+\Sum_{(\u,\v)\gdom(\s,\t)}r_{\u\v}n_{\u\v}$,
      for some $r_{\u\v}\in K$ and $0\neq c\in K$.
    \item If $\hat\psi_{\s\t}'$ is defined using a different choice of
      reduced expressions for $d(\s)$ and $d(\t)$ then
      $$\psi'_{\s\t}-\hat\psi'_{\s\t}
                =\sum_{(\u,\v)\gdom(\s,\t)}r_{\u\v}\psi'_{\u\v},$$
      where $r_{\u\v}\in K$ is non-zero only if $\res(\u)=\res(\s)$,
      $\res(\v)=\res(\t)$ and $\deg\u+\deg\v=\deg\s+\deg\t$.
  \end{enumerate}
\end{Prop}
}

Using Proposition~\ref{psi' properties}, and arguing exactly as in the
proof of Theorem~\ref{psi basis} we obtain the graded dual basis of $\H$.

\begin{Theorem}\label{psi' basis}
  The basis 
  $\set{\psi'_{\s\t}|\s,\t\in\Std(\blam) \text{ for }\blam\in\Multiparts}$
  is a graded cellular basis of $\H$.
\end{Theorem}

The basis $\{\psi'_{\s\t}\}$ is the \textbf{dual graded basis}
of~$\H$. We note that the unique anti-isomorphism of $\H$ which fixes
the homogeneous generators of $\H$ coincides with the graded
anti-isomorphisms coming from both the graded cellular basis and the
dual graded cellular basis, via (GC$_3$) of 
Definition~\ref{graded cellular def}.

As with the graded basis, the dual graded basis
restricts to give a graded cellular basis for the blocks of~$\H$.

\begin{cor}
  Suppose that $\beta\in Q_+$. Then
  $$\set{\psi'_{\s\t}|\s,\t\in\Std(\blam)
                           \text{ for }\blam'\in\Multiparts[\beta]}$$
  is a graded cellular basis of $\Hbeta$.
\end{cor}

\subsection{Graded symmetric algebras}
Recall that a \textbf{trace form} on a $K$-algebra $A$ is a
$K$-linear map $\tau\map AK$ such that $\tau(ab)=\tau(ba)$, for all
$a,b\in A$. The algebra~$A$ is \textbf{symmetric} if $A$ is
equipped with a non-degenerate symmetric bilinear form $\theta:
A\times A\rightarrow K$ which is associative in the following sense:
$$
\theta(xy,z)=\theta(x,yz),\quad\text{for all }x,y,z\in A.
$$
Define a trace form $\tau: A\rightarrow K$ on $A$ by
setting $\tau(a)=\theta(a,1)$ for any $a\in A$. Note that 
$\ker\tau$ cannot contain any non-zero left or right ideals because
$\theta$ is  non-degenerate. We leave the next result for the reader.

\begin{Lemma} \label{eqa} Suppose that $A$ is a finite dimensional $K$-algebra
  which is equipped with an anti-automorphism $\sigma$ of order~$2$. Then $A$
  is symmetric if and only if there is a non-degenerate symmetric bilinear
  form $\<\ ,\ \>\map{A\times A}K$ which is associative in
  the sense $\<ab,c\>=\<a,cb^\sigma\>$ for any $a,b,c\in A$.
\end{Lemma}

A graded algebra $A$ is a \textbf{graded symmetric} algebra if there exists
a homogeneous non-degenerate trace form $\tau\map AK$.  Apart from
providing a second graded cellular basis of $\H$, the dual graded basis
of~$\H$ is useful because we can use it to show that the algebras
$\Hbeta$, for $\beta\in Q_+$, are \textbf{graded symmetric algebras}. 

Following Brundan and Kleshchev~\cite[(3.4)]{BKW:GradedSpecht}, if
$\beta\in Q_+$ then the \textbf{defect} of $\beta$ is
$$\defect\beta=(\Lambda,\beta)-\frac12(\beta,\beta),$$
where $(\ ,\ )$ is the non-degenerate pairing on the root lattice
introduced in section~3.1. If $\ell=1$ then $\defect\beta$ is the
\textit{$e$-weight} of the block $\Hbeta$. If $\ell>1$ then
$\defect\beta$ coincides with Fayers~\cite{Fayers:AKweight} 
definition of weight for the algebras $\Hbeta$.

In what follows, the following result of Brundan, Kleshchev and Wang's will
be very important. (In~\cite[\Sect3]{BKW:GradedSpecht}, $\deg\s'$ is called
the \textit{codegree} of~$\s$.)

\begin{Lemma}[\protect{Brundan, Kleshchev and Wang~\cite[Lemma~3.12]{BKW:GradedSpecht}}]
    \label{defect}
  Suppose that $\bmu\in\Multiparts[\beta]$ and that $\s\in\Std(\bmu)$.
  Then $\deg\s+\deg\s'=\defect\beta$.
\end{Lemma}

To define the homogeneous trace form $\tau_\beta$ on $\Hbeta$ recall that,
by \cite{MM:trace} and \cite[Theorem~A2]{BK:HigherSchurWeyl}, $\H$ is a
symmetric algebra with induced trace form $\tau\map\H K$, where~$\tau$ is
the $K$-linear map determined by
$$\tau(L_1^{a_1}\dots L_n^{a_n}T_w)
              = \begin{cases}
                1,&\text{if }a_1=\dots=a_n=0, w=1 \text{ and }q\ne1,\\
                1,&\text{if }a_1=\dots=a_n=\ell-1, w=1 \text{ and }q=1,\\
                0,&\text{otherwise},
                \end{cases}
$$
where $0\le a_1,\dots,a_n<\ell$ and $w\in\Sym_n$.  In general, the
map $\tau$ is not homogeneous, however, we can use $\tau$ to define
a homogeneous trace form on $\Hbeta$ since $\Hbeta$ is a
subalgebra of $\H$.

\begin{Defn}[Homogeneous trace]
  Suppose that $\beta\in Q_+$. Then $\tau_\beta\map{\Hbeta}K$ is the
  map which on a homogeneous element $a\in\Hbeta$ is given by
  $$\tau_\beta(a) = \begin{cases}
    \tau(a),&\text{if }\deg(a)=2\defect\beta,\\
    0,&\text{otherwise.}
  \end{cases}$$
\end{Defn}

It is an easy exercise to verify that $\tau_\beta$ is a trace form on
$\Hbeta$. By definition, $\tau$ is homogeneous of degree $-2\defect\beta$.
To show that $\tau_\beta$ is induced from a non-degenerate symmetric
bilinear form on $\Hbeta$ we need the following fact.

\begin{Lemma}    \label{facts}
    Suppose that $\a,\b\in\Std(\bmu)$ and $\c,\d\in\Std(\bnu)$, for
    $\bmu,\bnu\in\Multiparts[\beta]$. Then $m_{\a\b}n_{\c\d}\ne0$ only if
    $\c'\gedom\b$. Further, there exists a non-zero scalar $c_\blam\in
    K$, which depends only on $\blam$, such that
    $$\tau(m_{\a\b}n_{\d\c})=\begin{cases}
    c_\blam,&\text{if }(\c',\d')=(\a,\b),\\
    0,&\text{if }(\c',\d')\not\gedom(\a,\b).
    \end{cases}$$
\end{Lemma}

\begin{proof}
  In the non-degenerate case this is a restatement of \cite[Lemma 5.4 and
  Theorem 5.5]{M:tilting}, which reduces the calculation of this trace to
  \cite[Theorem~5.9]{M:gendeg} which gives the trace of a certain generator
  of the Specht module. 
  
  We sketch the proof in the degenerate case. The arguments of
  \cite{M:tilting} can be repeated word for word using the cellular basis
  framework for the degenerate cyclotomic Hecke algebras given in
  \cite[\Sect6]{AMR}. The main difference in the degenerate case is that
  the arguments from~\cite{M:gendeg} simplify. In particular, using the
  notation of \cite{M:gendeg}, in the degenerate case we can replace the
  complicated \cite[Lemma~5.8]{M:gendeg} with the simpler statement that 
  $$
  T_{w_{\lam}}u_{\lam'}^{-} = L_{a_2+1,n}(Q_1)\cdots L_{a_r+1,n}(Q_{r-1})
  T_{w_{\lam}} + \epsilon,
  $$
  where $\epsilon$ is a linear combination of some elements of the form
  $L_1^{c_1}L_2^{c_2}\cdots L_n^{c_n}T_w$ such that $0\leq c_i<\ell, w\in
  S_n$ and at least one of these $c_i$ is strictly less than $\ell-1$. This
  is easily proved using the relation $T_iL_i-L_{i+1}T_i=-1$, for $1\le
  i<n$. Once this change is made the analogue of
  \cite[Theorem~5.9]{M:gendeg} in the degenerate case can be proved
  following the arguments of \cite{M:gendeg}.
\end{proof}

Define a homogeneous bilinear form $\<\ ,\ \>_\beta$ on $\Hbeta$ of degree
$-2\defect\beta$ by 
$$\<a,b\>_\beta=\tau_\beta(ab^*).$$
By definition, $\<\ ,\ \>_\beta$ is symmetric and associative in the sense that
$\<a,bc\>_\beta=\<ac^*,b\>_\beta$ for any $a,b,c\in\Hbeta$.

\begin{Theorem}\label{pairing}
    Suppose that $\beta\in Q_+$ and that
    $\blam,\bmu\in\Multiparts[\beta]$. If $\s,\t\in\Std(\blam)$ and
    $\u,\v\in\Std(\bmu)$ then
    $$\<\psi_{\s\t},\psi'_{\u\v}\>_\beta
                  = \begin{cases}
                      u,&\text{if }(\u',\v')=(\s,\t),\\
              0,&\text{if }(\u',\v')\not\gedom(\s,\t),
                    \end{cases}$$
    for some non-zero scalar $u\in K$ which depends on $\s$ and $\t$.
\end{Theorem}

\begin{proof}
   By Lemma~\ref{psi triangular} and Proposition~\ref{psi' properties}(c),
   there exist non-zero scalars $c,c'\in K$ and $r_{\a\b},r'_{\d\c}\in K$
   such that
   $$ \psi_{\s\t}{\psi'_{\v\u}}
       =\Big(cm_{\s\t}+\sum_{(\a,\b)\gdom(\s,\t)}r_{\a\b}m_{\a\b}\Big)
         \Big(c'n_{\v\u}+\sum_{(\d,\c)\gdom(\v,\u)}r'_{\d\c}n_{\d\c}\Big).
     \leqno(\dag)
   $$
   Therefore, $\<\psi_{\s\t},\psi'_{\u\v}\>_\beta=0$ unless
   $\v'\gedom\t$ by Lemma~\ref{facts}. Now,
   $$\<\psi_{\s\t},\psi'_{\u\v}\>_\beta
          =\tau_\beta(\psi_{\s\t}\psi'_{\v\u}) 
          =\tau_\beta(\psi'_{\v\u}\psi_{\s\t}) 
          =\tau_\beta(\psi_{\t\s}\psi'_{\u\v}) 
          =\<\psi_{\t\s},\psi'_{\v\u}\>_\beta,$$
    where we have used the easily checked fact that
    $\tau_\beta(h)=\tau_\beta(h^*)$ for the third equality. Combined with
    $(\dag)$, this shows that $\<\psi_{\s\t},\psi'_{\u\v}\>_\beta=0$ unless
    $(\u',\v')\gedom(\s,\t)$.

   To complete the proof it remains to consider the case when
   $(\u',\v')=(\s,\t)$. By Lemma~\ref{facts}, (\dag) now reduces to the
   equation $\psi_{\s\t}{\psi'_{\t'\s'}}=cc'm_{\s\t}n'_{\t'\s'}$. By
   Lemma~\ref{psi degrees}, Proposition~\ref{psi' properties}(b) and
   Lemma~\ref{defect}, we have 
   $$\deg(\psi_{\s\t}\psi'_{\t'\s'})
           =\deg\s+\deg\t+\deg\s'+\deg\t'=2\defect\beta,$$
   Therefore, we can replace $\tau_\beta$ with $\tau$ and use Lemma~\ref{facts} 
   to obtain
   $$\tau_\beta(\psi_{\s\t}\psi'_{\t'\s'})
           =\tau(\psi_{\s\t}\psi'_{\t'\s'})
	   =cc'\tau(m_{\s\t}n_{\t'\s'})
           =cc'c_\blam.$$
   As $cc'c_\blam\ne0$ this completes the proof.
\end{proof}

Applying Lemma \ref{eqa}, we deduce that $\Hbeta$
is a graded symmetric algebra. This was conjectured by Brundan and
Kleshchev~\cite[Remark~4.7]{BK:GradedDecomp},

\begin{cor} \label{gradsy}
    Suppose that $\beta\in Q_+$. Then $\Hbeta$ is a graded symmetric
    algebra with  homogeneous trace form $\tau_\beta$ of
    degree $-2\defect\beta$.
\end{cor}

We remark that the two graded bases $\{\psi_{\s\t}\}$ and
$\{\psi'_{\u\v}\}$ are almost certainly not dual with respect to $\<\
,\ \>_\beta$. We call $\{\psi'_{\u\v}\}$ the \textit{dual} graded
basis because  Theorem~\ref{pairing} shows that these two bases are
dual modulo more dominant terms. As far as we are aware, if $\ell>2$
then there are no known pairs of dual bases for~$\H$, even
in the ungraded case.

\subsection{Dual graded Specht modules}
Using the graded cellular basis $\{\psi_{\s\t}\}$ we defined the graded
Specht module~$S^\blam$. Similarly, if $\blam\in\Multiparts$ then the
\textbf{dual graded Specht module} $S_\blam$ is the graded cell module
associated with $\blam$, via Definition~\ref{graded cells}, using the dual
graded basis $\{\psi_{\s\t}'\}$.  Thus, $S_\blam$ has a homogeneous basis
$\set{\psi'_\s|\s\in\Std(\blam)}$, with the action of $\H$ being induced by
its action on the dual graded basis.

By \cite[Cor.~5.7]{M:tilting}, it was shown that $\underline{S}^\blam$ and
$\underline{S}_\blamp$ are dual to each other with respect to the
contragredient duality induced on $\H\Mod$ by the cellular algebra
anti-isomorphism defined by the standard cellular basis $\{m_{\s\t}\}$. We
generalize this result to the graded setting.

Let $\Hpblam
   =\<\psi_{\u\v}\mid\u,\v\in\Std(\bmu)\text{ where }\bmu\gdom\blam\>_K$ be the graded two-sided ideal of $\H$ spanned by the
elements of the cellular basis $\{\psi'_{\u\v}\}$ of more dominant shape.
Then $\Hpblam$ is also spanned by the elements $\{n_{\u\v}\}$, where
$\u,\v\in\Std(\bmu)$ and $\bmu\gdom\blam$ by 
Proposition~\ref{psi' properties}(c).

\begin{Prop}
    Suppose that $\blam\in\Multiparts[\beta]$. Then
    $S^\blam\cong S_\blamp^\circledast\<\defect\beta\>$
    as graded $\Hbeta$-modules.
\end{Prop}

\begin{proof}By Theorem~\ref{pairing} the graded two-sided ideals 
  $\HH^{\gdom\blam}_\beta$ and $\HH^{\prime\gdom\blamp}_\beta$
  of~$\Hbeta$ are orthogonal with respect to the trace form 
  $\<\ ,\ \>_\beta$. By construction 
  $S^\blam\<\deg\tlam\>\cong(\psi_{\tlam\tlam}+\Hblam)\H$ and
  $S_\blamp\<\deg\tllamp\>\cong(\psi'_{\tllamp\tlamp}+\Hpblam[\blam])\H$,
  where $\tllamp=(\tlam)'$. Therefore, $\<\ ,\ \>_\beta$ induces a
  homogeneous associative bilinear form 
  $$\<\ ,\ \>_{\beta,\blam}
      \map{S^\blam\<\deg\tlam\>\times S_\blamp\<\deg\tllamp\>}K;
       \<a+\Hblam,b+\Hpblam[\blamp]\>_{\beta,\blam}=\<a,b\>_\beta.$$
In particular, if $\s,\t'\in\Std(\blam)$ then, by Theorem~\ref{pairing},
$$\<\psi_{\tlam\s}+\Hblam,\psi'_{\tllamp\t}+\Hpblam[\blamp]\>_{\beta,\blam}
           =\begin{cases} u, &\text{if } \s=\t',\\
                          0,&\text{unless}\,\,\t'\gedom\s,
            \end{cases}$$
for some $0\ne u\in K$. Therefore,  $\<\ ,\ \>_{\beta,\blam}$ is a
homogeneous non--degenerate pairing of degree $-2\defect\beta$ and, since taking duals reverses the grading,
$$S^\blam\cong S_\blamp^\circledast\<2\defect\beta-\deg\tllamp-\deg\tlam\>
             =S_\blamp^\circledast\<\defect\beta\>,$$
since $\defect\beta=\deg\tlam+\deg\tllamp$
by Lemma~\ref{defect}.
\end{proof}

During the proof of Theorem~\ref{pairing} we showed that
$m_{\s\t}n_{\t'\s'}=c\psi_{\s\t}\psi'_{\t'\s'}$, for some non-zero
constant $c\in K$. Hence, we have the following interesting fact.

\begin{cor}[of Theorem~\ref{pairing}]
  Suppose that $\blam\in\Multiparts[\beta]$ and that
  $\s,\t\in\Std(\blam)$. Then
  $m_{\s\t}n_{\t'\s'}$ %=c\psi_{\s\t}\psi'_{\t'\s'},$ where $0\ne c\in K$, 
  is a homogeneous element of $\H$ of degree $2\defect\beta$.
\end{cor}

Let $\blam\in\Multiparts[\beta]$. Recall that by definition, 
$e_{\blam}=e(\bi^{\tlam})$ and $e'_{\blam'}=e(\bi^{\t_{\blam}})$,
where $\tllam=(\t^{\blamp})'$. Let $w_\blam=d(\tllam)$ and define
$z_\blam = m_\blam T_{w_\blam}n_\blamp$.

\begin{cor}\label{zlam} Suppose that $\blam\in\Multiparts[\beta]$. Then
$$
z_{\blam}=e_{\blam}z_{\blam}e'_{\blam'}
         =ce_{\blam}y_{\blam}\psi_{w_{\blam}}y_{\blam'}
	 =cy_{\blam}\psi_{w_{\blam}}y_{\blam'}e'_{\blam'},
$$
for some $0\neq c\in K$. In particular, $z_{\blam}$ is a homogeneous
element of $\H$ of degree $\defect\beta+\deg(\t^{\blam})+\deg(\t^{\blam'})$. 
\end{cor}

\begin{proof} By Corollary \ref{ylam} and (\ref{dualep}) there exist 
    $0\ne c\in K$ such that
$$
e_{\blam}y_{\blam}\psi_{w_{\blam}}\equiv ce_{\blam}m_{\tlam\tllam}
       +\sum_{\substack{\t\in\Std(\blam)\\\ell(d(\t))<\ell(w_{\blam})}}
                 a_{\t}e_{\blam}m_{\tlam\t}\pmod{\Hblam},
$$
for some $a_{\t}\in K$. Further, $e'_{\blam'}y_{\blam'}\equiv
c'e'_{\blam'}n_{\blam'}\pmod{{\HH'_n}^{\rhd\blam'}}$, for some non-zero
$c'\in K$,  by Proposition~\ref{copositive expansion}. By definition $\t\gedom\tllam$ for all $\t\in\Std(\blam)$, so if $\t\ne\tllam$ then
$m_{\tlam\t}n_{\blam'}=0$ by Lemma~\ref{facts} since
$(\t^\blamp)'=\tllam\not\gedom\t$.  Hence, multiplying these two equations
together gives the Corollary.
\end{proof}

There may well be a more direct proof of the last two results because
these elements are already well-known in the representation theory
of~$\H$. Note that
$$m_{\s\t}n_{\t'\s'}
          =T_{d(\s)^{-1}}m_\blam T_{d(\t)}T_{d(\t')^{-1}} n_\blamp T_{d(\s')}
          =T_{d(\s)^{-1}} z_\blam T_{d(\s')},$$
because $d(\t)d(\t')^{-1}=w_\blam$, with the lengths adding; see, for
example, \cite[Lemma~5.1]{M:tilting}. It follows from
\cite[Prop.~4.4]{M:gendeg} that 
$(T_{d(\s)^{-1}} z_\blam T_{d(\s)})^2=rT_{d(\s)^{-1}} z_\blam T_{d(\s)}$, 
for some $r\in K$, such that $r\ne0$ if and only if the Specht module
$S^\blam$ is projective. If $r=0$ then these elements are
nilpotent and they belong the radical of~$\H$.  We invite the reader
to check that the map
$$S_\blamp\<\defect\beta+\deg\tlam\>
       \overset{\sim}{\longrightarrow}z_\blam\H; 
               \psi'_\t\mapsto z_\blam\psi'_{d(\t)},$$ 
for $\t\in\Std(\blamp)$, is a isomorphism of graded $\H$-modules.
Similarly, there is a graded isomorphism
$S^\blam\<\defect\beta+\deg\t^{\blam'}\>
\overset{\sim}{\longrightarrow}n_{\blam'}T_{w_{\blam'}}m_{\blam}\H$. 
By Corollary~\ref{zlam}, $z_\blam^*=ce_\blamp\psi_{w_\blamp}\eblam$ is
homogeneous of degree $\defect\beta+\deg(\t^{\blam})+\deg(\t^{\blam'})$,
for some non-zero $c\in K$. Arguing as in Corollary~\ref{zlam} shows that
$z_\blam^*=n_\blamp T_{w_\blamp} m_\blam$. Consequently, on the elements
$z_\blam$, for $\blam\in\Multiparts$, the graded cellular 
anti-automorphism~$*$ of $\H$ coincides with the ungraded cellular algebra
anti-isomorphism which is induced by the standard basis  $\{m_{\u\v}\}$
of~$\H$.

%%%%%%%%%%%%%%%%%%%%%%%%%%%%%%%%%%%%%%%%%%%%%%%%%%%%%%%%%%%%%%%%%%%%%%%
\appendix
\def\theequation{\Alph{section}\arabic{equation}}
\section{One dimensional homogeneous representations}
Using Theorem~\ref{psi basis} it is straightforward to give an
explicit homogeneous basis for the one dimensional two-sided ideals
of~$\H$. In this appendix, which may be of independent interest, we
give a proof of this result without appealing to 
Theorem~\ref{psi basis}. We consider only the non-degenerate case here
and leave the easy modifications required for the degenerate case to
the reader.

We remark that it is possible to prove an analogue of
Theorem~\ref{psi basis} using the ideas in this appendix. However,
using these techniques we were only able to show that the basis
$\{\psi_{\s\t}\}$ was a graded cellular basis with respect to the
\textit{lexicographic} order on~$\Multiparts$.

\begin{Defn}
   Suppose that $1\le s\le e$ and $(\Lambda,\alpha_s)>0$ and set
   \begin{align*}
     u_{n,s}=\prod_{i\in I}
   \((L_1-q^i)\dots(L_n-q^i)\)^{(\Lambda,\alpha_i)-\delta_{is}},
          \\
     x_{(n)}=\sum_{w\in\Sym_n}T_w
          \quad\text{and}\quad
          x'_{(n)}=\sum_{w\in\Sym_n}(-q)^{-\ell(w)}T_w.
   \end{align*}
   Finally, define $z_n^{+,s}=u_{n,s}x_{(n)}$ and
   $z_n^{-,s}=u_{n,s}x'_{(n)}$, for
   $1\le s\le e$.
\end{Defn}

The following result is well-known and easily verified.

\begin{Lemma}\label{one dimensional}
  Suppose that $1\le s\le e$ and that $\eps\in\{+,-\}$. Then
  \begin{align*}
     T_w\zns&=\zns T_w
           =(-1)^{\frac12(1-\eps1)\ell(w)}q^{\frac12(1+\eps1)\ell(w)} \zns,\\
    L_k\zns&=\zns L_k=q^{s+\eps(k-1)}\zns,
  \end{align*}
   for all $w\in\Sym_n$ and $1\le k\le n$.
   In particular, $Kz_n^{\pm,s}$ is a one dimensional two-sided ideal
   of $\H$. Moreover, every one dimensional two-sided ideal is isomorphic to
   $K\zns$, for some~$s$, and 
   $$K\zns=\Biggl\{h\in\H\Biggm|\begin{matrix}
          T_0h=q^{s}h=hT_0\text{ and}\\
          T_ih=hT_i=(-1)^{\frac12(1-\eps1)}q^{\frac12(1+\eps1)}h
            \text{ for }1\le i<n\end{matrix}\Biggr\}.$$
\end{Lemma}

The following result contains the simple idea which drives this
appendix.

\begin{Prop}\label{homogeneous}
  Suppose that $Kz$ is a two sided ideal $\R$, for some
  non-zero element $z\in\H$. Then $z$ is homogeneous.
\end{Prop}

\begin{proof}
  Write $z=\sum_{i\in\Z} z_i$, where $z_i$ is a homogeneous element of
  degree $i$, for each $i\in\Z$, with only finitely many $z_i$ being
  non-zero. Let $h\in\H$ be any homogeneous element. Then
  $hz=fz$, for some $f\in K$, so that
  $$\sum_{i\in\Z}fz_i=hz=\sum_{i\in\Z}hz_i.$$
  By assumption, either $hz_i=0$ or $\deg(hz_i)=\deg h+\deg z_i$, for each
  $i$.  Therefore, if $\deg h>0$ and $hz\ne0$ then $hz_i=fz_j$ for some
  $j>i$, which is a contradiction since this forces $hz=fz$ to have fewer
  homogeneous summands than~$z$.  Therefore, $hz=0$ if $\deg h>0$.
  Similarly, $hz=0$ if $\deg h<0$.  Therefore, for any $h\in\H$ we have
  that $hz_i=fz_i$, for all $i\in\Z$, so that $z_i=z_n^{\pm,s}$, for some
  $s$ by Lemma~\ref{one dimensional}. Since the non-zero $z_i$ have
  different degrees they must be linearly independent, so it follows from
  Lemma~\ref{one dimensional} that $z=z_i$ for a unique~$i$.
  In particular, $z$ is homogeneous as claimed.
\end{proof}

The following definition will be used to give the degree of the elements
$z^\eps_{n,s}$ and to explicitly describe them as a product of the
homogeneous generators of $\H$.

We extend our use of the Kronecker delta by writing, for any statement $S$,
$\delta_{S}=1$ if $S$ is true and $\delta_S=0$ otherwise.

\begin{Defn}[cf. Definition~\ref{positive}]
  Suppose that $1\le s\le e$ and let $\eps\in\{+,-\}$. Let
  $\ins=(i^{\eps,s}_1,\dots,i^{\eps,s}_n)\in I^n$, where
  $i^{\eps,s}_k=s+\eps (k-1)\pmod e$. For $1\le k\le n$ set
  $$d^{\eps,s}_k=\set{1\le t\le\ell|i^{\eps,s}_k=t
           \text{ and }(\Lambda,\alpha_t)>\delta_{st}}+\delta_{e|k}.$$
  Finally, define $y^{\eps,s}_n=\prod_{k=1}^n y_k^{d^{\eps,s}_k}$.
\end{Defn}

Brundan, Kleshchev and Wang~\cite[(4.5)]{BKW:GradedSpecht}
note that the natural embedding $\H\hookrightarrow\H[n+1]$ is an
embedding of graded algebras. Explicitly, the  graded embedding is
determined by
\begin{equation}\label{embedding}
  \psi_s\mapsto\psi_s,\qquad y_r\mapsto y_r,\qquad \text{and}\qquad
  e(\bi)\mapsto\sum_{j\in I}e(\bi\vee j),
\end{equation}
where $1\le r\le n$, $1\le s<n$, $\bi\in I^n$ and 
$\bi\vee i=(i_1,\dots,i_n,i)$. 

In what follows we need an explicit formula for the elements
$P_r(\bi)$, where $1\le r<n$ and $\bi\in I^n$,  which were discussed briefly just before Theorem~\ref{BK
main}. To define these, for $\bi\in I^n$ set
$$
y_r(\bi) := q ^{i_r}(1-y_r) \in K\llbracket y_1,\dots,y_n\rrbracket,
$$
and, recalling that $q\ne1$, define formal power series
$P_r(\bi) \in K\llbracket y_r,y_{r+1}\rrbracket$ by setting
$$
P_r(\bi)=\begin{cases}
1 & \text{if $i_r=i_{r+1}$},\\
(1-q )\left(1-y_r(\bi) y_{r+1}(\bi)^{-1}\right)^{-1} & \text{if $i_r\neq i_{r+1}$}.
\end{cases}
$$
By a small generating function exercise, if $i_r\ne i_{r+1}$ then
\begin{equation}\label{P expansion}
  P_r(\bi)=\frac{1-q}{1-q^{i_r-i_{r+1}}}
  \Bigg\{1+\sum_{k\ge1}\frac{q^{i_r-i_{r+1}}
  (y_{r+1}-y_r)(y_{r+1}-q^{i_r-i_{r+1}}y_r)^{k-1}}{(1-q^{i_r-i_{r+1}})^k}
         \Bigg\}.
\end{equation}

We can now explicitly describe $\zns$ as a product of homogeneous
elements and hence determine its degree.

\begin{Theorem}\label{zns}
  Suppose that $1\le s\le e$, $(\Lambda,\alpha_s)>0$ and that $\eps\in\{+,-\}$. Then
  $$\zns=Ce(\ins)y^{\eps,s}_n,
  $$
  for some non-zero constant $C\in K$. In particular,
  $\deg\zns=2(\dns[1]+\dots+\dns)$.
\end{Theorem}

\begin{proof} As $K\zns$ is a two-sided ideal we have that $e(\ins)\zns
  e(\ins)\in K\zns$. Further, it is well-known and easy to check
  (\textit{cf.} \cite[\Sect4]{M:tilting}), that
  $K\zns\cong S(\blam)$, where
  $\blam=(\lambda^{(1)},\dots,\lambda^{(\ell)})$ and
  $$\lambda^{(t)}=\begin{cases}
    (n),&\text{if $t=s$ and $\eps=+$},\\
    (1^n),&\text{if $t=s$ and $\eps=-$},\\
     (0),&\text{otherwise}.
   \end{cases}$$
  Therefore, as $\ins=\ilam$ it follows from the construction of the graded
  Specht modules in section~5.2
  (or~\cite[Theorem~4.10]{BKW:GradedSpecht}), that $\zns e(\ins)\ne0$, so
  we see that $\zns=e(\ins)\zns=\zns e(\ins)=e(\ins)\zns e(\ins)$ as
  claimed.

  It remains to write $\zns$ as a product of homogeneous elements. To
  ease the notation we treat only the case when $\eps=+$ and we write
  $z_n=z_n^{\eps,s}$, $\bi_n=\ins$ and $d_n=\dns$. The case when $\eps=-$
  follows by exactly the same argument (and, in fact, the same
  constants appear below), the only difference is that the products
  $T_{n-1}\dots T_j$ must be replaced by
  $(-q)^{j-n}T_{n-1}\dots T_j$ below.

  Suppose, first, that $n=1$. By definition,
  $d_1=(\Lambda,\alpha_s)-1$. Recall that
  $L_1=\sum_\bi q^{i_1}(1-y_1)e(\bi)$ by Theorem~\ref{BK main}. Therefore, we have
  \begin{align*}
  z_1e(\bi_n)&=\prod_{t\in I}
  (L_1-q^{t})^{(\Lambda,\alpha_t)-\delta_{st}}e(\bi_n)
  = \prod_{t\in I}(q^{s}-q^{t}-q^{s}y_1)^{(\Lambda,\alpha_t)-\delta_{st}}e(\bi_n)\\
       &=\prod_{t\ne s} (q^{s}-q^{t}-q^{s}y_1)^{(\Lambda,\alpha_t)} e(\bi_n)
             \cdot(-q^{s}y_1)^{(\Lambda,\alpha_{s})-1}e(\bi_n)\\
       &=\prod_{t\ne s} (q^{s}-q^t)^{(\Lambda,\alpha_t)}
             \cdot(-q^{s}y_1)^{(\Lambda,\alpha_{s})-1}e(\bi_n)
  \end{align*}
  where the last equality follows because the `cyclotomic relation'
  $y_1^{(\Lambda,\alpha_{s})}e(\bi_n)=0$, holds in $\R$.
  Thus, the Theorem holds when $n=1$.

Now suppose that $n>1$ and that the Theorem holds for smaller $n$.
Then, using the definitions,
\begin{align*}
  z_n&=e(\bi_n)\prod_{t\in I}(L_n-q^{t})^{(\Lambda,\alpha_t)-\delta_{st}}\cdot z_{n-1}\cdot
  \Big(1+\sum_{j=1}^{n-1}T_{n-1}\dots T_j\Big)
  e(\bi_n)\\
     &=\prod_{t\in I}(L_n-q^t)^{(\Lambda,\alpha_t)-\delta_{st}}\cdot e(\bi_n)z_{n-1}\cdot
  \Big(1+\sum_{j=1}^{n-1}T_{n-1}\dots T_j\Big)
  e(\bi_n).
\end{align*}

By induction and (\ref{embedding}), there exists a scalar non-zero $C\in K$ such that
\begin{align*}
e(\bi_n)z_{n-1}=z_{n-1} e(\bi_{n})
&=Cy^{\eps,s}_{n-1}\prod_{i\in I}e(\bi_{n-1}\vee i)\cdot
e(\bi_n)\\
&=Cy^{\eps,s}_{n-1}e(\bi_n)
\end{align*}
Let $d_n'=d_n-\delta_{e|n}$. Then there exist constants $C_a'\in K$,
for $a\ge d_n'$, such that
\begin{align*}
\prod_{t\in I}(L_n-&q^{t})^{(\Lambda,\alpha_t)-\delta_{st}}\cdot e(\bi_n)z_{n-1}\\
&=C\prod_{t\in I}\(q^{s+(n-1)}(1-y_n)-q^{t}\)^{(\Lambda,\alpha_t)-\delta_{st}}
          \cdot y^{\eps,s}_{n-1} e(\bi_n)\\
&=e(\bi_n)y^{\eps,s}_{n-1}\sum_{a\ge d_n'}C_ay_n^a,
\end{align*}
with $C_{d_n'}=C(-q)^{(s+(n-1))d_n'}
\prod_{t}(q^{s+(n-1)}-q^{t})^{(\Lambda,\alpha_t)-\delta_{st}}$, where the product
is over those $t\in I$ with $t\not\equiv s+(n-1)\pmod{e\Z}$.
%and $$C_{d_n'+1}=-C_{d_n'}
%   \sum_{\substack{t\not\equiv s+(n-1)\\(\Lambda,\alpha_t)>0}}
%   \frac{q^t}{(q^{s+(n-1)}-q^{t})}.$$
In particular, $C_{d_n'}\ne0$.
% and $C_{d_n'+1}\ne0$.
Next, recall from Theorem~\ref{BK main} that
$$T_ke(\bi_n)=\(\psi_kQ_k(\bi_n)-P_k(\bi_n)\)e(\bi_n),$$
for $1\le k\le n$. Applying
the relations in~(\ref{relations}), if $1\le k_1<\dots<k_p<n$ then
$$e(\bi_n)\psi_{k_p}\dots\psi_{k_1}e(\bi_n)
=\psi_{k_p}\dots\psi_{k_1}e(s_{k_1}\dots s_{k_p}\cdot\bi_n) e(\bi_n)
=0.$$
Moreover, by the proof of  Proposition~\ref{homogeneous} we know that
$z_{n-1}y_i=0$, for $1\le i<n$. Therefore, when we expand $P_j(\bi_n)$ as a
power series in $K\llbracket y_1,\dots,y_n\rrbracket$ only those terms in
$K\llbracket y_n\rrbracket$ contribute to $z_n$. Putting all of this
together we find that
$$
z_n=e(\bi_n)y^{\eps,s}_{n-1}\sum_{a\ge d_n'}C_a'y_n^a
$$
for some $C_a'\in K$. Notice that only one of these terms can survive since $z_n$ is
homogeneous by Proposition~\ref{homogeneous}. By (\ref{P expansion})
the constant term of $P_j(\bi_n)$ is
$-(1-q)/(1-q^{-1})=q$, so
$$ \frac{C_{d_n'}'}{C_{d_n'}}
       =1+\sum_{j=1}^{n-1}q^t
        =1+q+\dots+q^{n-1}.
$$
Therefore, $C_{d_n'}'\ne0$ if and only if $e\nmid n$, which is exactly the
case when $d_n'=d_n$ so the Theorem holds when $e\nmid n$.

Finally, suppose that $e|n$. Then $C_{d_n'}'=0$, by what we have just
shown, and $d_n=d_n'+1$, so we need to show that $C_{d_n'+1}'\ne0$.
This time the degree one term of $P_n(\bi_n)$ and the degree zero
terms of $P_j(\bi_n)$, for $1\le j<n$, contribute to~$C_{d_n'+1}'$.
Using (\ref{P expansion}) again, we find that
$$\frac{C_{d_n'+1}'}{C_{d_n'}} =\frac
q{q-1}\(q+q^2+\dots+q^{n-1})=\frac q{1-q}\ne0.$$ This completes the
proof of the Theorem.
\end{proof}

We remark that we do not know how to prove Theorem~\ref{zns} using the
relations directly. One problem, for example, is that it is not clear
from the proof of Theorem~\ref{zns} that $C_{d_n'+1}=0$ when $e\nmid
n$ -- note that if $C_{d_n'+1}\ne0$ then $z_n$ would not be homogeneous
since $C_{d_n'}\ne0$ when $e\nmid n$. We are able to prove
Theorem~\ref{zns} only because we already know that $z_n$ is
homogeneous by Proposition~\ref{homogeneous}.

\section*{Acknowledgments}
This research was supported by the Australian Research Council. The
authors thank Jonathan Brundan, Alexander Kleshchev and Sin\'ead Lyle for
discussions.

%%%%%%%%%%%%%%%%%%%%%%%%%%%%%%%%%%%%%%%%%%%%%%%%%%%%%%%%
%\bibliography{papers}

\end{document}